\documentclass[12pt]{amsart}
\usepackage{todonotes}

\usepackage{tikz,amsthm,amsmath,amstext,amssymb,amscd,epsfig,euscript, mathrsfs, dsfont,pspicture,multicol,graphpap,graphics,graphicx,times,enumerate,subfig,sidecap,wrapfig,color,pict2e}

\usepackage[colorlinks,citecolor=cyan,pagebackref,hypertexnames=false]{hyperref}

\usepackage{enumerate}

 \numberwithin{equation}{section}

\def\bB{{\mathbb{B}}}

\def\R{{\mathbb{R}}}

\def\bZ{{\mathbb{Z}}}

\def\bN{{\mathbb{N}}}

\def\cA{{\mathscr{A}}}
\def\cB{{\mathscr{B}}}
\def\cC{{\mathscr{C}}}
\def\cD{{\mathscr{D}}}

\def\cG{{\mathscr{G}}}
\def\cH{{\mathscr{H}}}
\def\cI{{\mathscr{I}}}

\def\cN{{\mathscr{N}}}

\def\one{\mathds{1}}

\def\ve{\varepsilon}
\def\vp{\varphi}

\def\lec{\lesssim}
\def\gec{\gtrsim}

\DeclareMathOperator{\diam}{diam}
\def\dist{\mathop\mathrm{dist}} 						
\def\supp{\mathop\mathrm{supp}}					

\newcommand{\ps}[1]{\left( #1 \right)}

\def\XXint#1#2#3{{\setbox0=\hbox{$#1{#2#3}{\int}$ }
\vcenter{\hbox{$#2#3$ }}\kern-.58\wd0}}

\theoremstyle{plain}

\newtheorem{theorem}{Theorem}

\newtheorem{corollary}[theorem]{Corollary}

\newtheorem{lemma}[theorem]{Lemma}

\newtheorem{proposition}[theorem]{Proposition}

\theoremstyle{definition}

\newtheorem{definition}[theorem]{Definition}
\newtheorem{remark}[theorem]{Remark}

\numberwithin{equation}{section}
\numberwithin{theorem}{section}

  \DeclareFontFamily{U}{mathb}{\hyphenchar\font45} 
\DeclareFontShape{U}{mathb}{m}{n}{
      <5> <6> <7> <8> <9> <10> gen * mathb
      <10.95> mathb10 <12> <14.4> <17.28> <20.74> <24.88> mathb12
      }{}
\DeclareSymbolFont{mathb}{U}{mathb}{m}{n}

\DeclareMathSymbol{\toitself}{3}{mathb}{"FD}  

\makeatletter
\def\@tocline#1#2#3#4#5#6#7{\relax
  \ifnum #1>\c@tocdepth 
  \else
    \par \addpenalty\@secpenalty\addvspace{#2}%
    \begingroup \hyphenpenalty\@M
    \@ifempty{#4}{%
      \@tempdima\csname r@tocindent\number#1\endcsname\relax
    }{%
      \@tempdima#4\relax
    }%
    \parindent\z@ \leftskip#3\relax \advance\leftskip\@tempdima\relax
    \rightskip\@pnumwidth plus4em \parfillskip-\@pnumwidth
    #5\leavevmode\hskip-\@tempdima
      \ifcase #1
       \or\or \hskip 1em \or \hskip 2em \else \hskip 3em \fi%
      #6\nobreak\relax
    \dotfill\hbox to\@pnumwidth{\@tocpagenum{#7}}\par
    \nobreak
    \endgroup
  \fi}
\makeatother

\begin{document}

\title[Weak lower density]{The weak lower density condition and uniform rectifiability}


\author[Azzam]{Jonas Azzam}

\address{Jonas Azzam\\
School of Mathematics \\ University of Edinburgh \\ JCMB, Kings Buildings \\
Mayfield Road, Edinburgh,
EH9 3JZ, Scotland.}
\email{j.azzam "at" ed.ac.uk}

\author[Hyde]{Matthew Hyde}

\address{Matthew Hyde\\
School of Mathematics \\ University of Edinburgh \\ JCMB, Kings Buildings \\
Mayfield Road, Edinburgh,
EH9 3JZ, Scotland.}
\email{m.hyde "at" ed.ac.uk}

\begin{abstract}
We show that an Ahlfors $d$-regular set $E$ in $\R^{n}$ is uniformly rectifiable if the set of pairs $(x,r)\in E\times (0,\infty)$ for which there exists $y \in B(x,r)$ and $0<t<r$ satisfying $\cH^{d}_{\infty}(E\cap B(y,t))<(2t)^{d}-\ve(2r)^d$ is a Carleson set for every $\ve>0$.

To prove this, we generalize a result of Schul by proving, if $X$ is a $C$-doubling metric space, $\ve,\rho\in (0,1)$, $A>1$, and $X_{n}$ is a sequence of maximal $2^{-n}$-separated sets in $X$, and $\cB=\{B(x,2^{-n}):x\in X_{n},n\in \bN\}$, then 
\[
\sum \left\{r_{B}^{s}: B\in \cB, \frac{\cH^{s}_{\rho r_{B}}(X\cap AB)}{(2r_{B})^{s}}>1+\ve\right\} \lec_{C,A,\ve,\rho,s} \cH^{s}(X).
\]
This is a quantitative version of the classical result that for a metric space $X$ of finite $s$-dimensional Hausdorff measure,  the upper $s$-dimensional densities are at most $1$ $\cH^{s}$-almost everywhere.

\end{abstract}

\maketitle

\tableofcontents

\section{Introduction}
A classical fact from geometric measure theory is that, if the lower densities of a set of finite $\cH^{d}$-measure are close enough to $1$, then the set is $d$-rectifiable. Recall that a metric space $X$ is {\it $d$-rectifiable} if it may be covered up to a set of zero $d$-dimensional Hausdorff measure  (denoted $\cH^{d}$) by Lipschitz images of subsets of $\mathbb{R}^{d}$. We define the lower and upper $d$-dimensional densities of a set $E$ at a point $x$ to be 
\[
\Theta^{d}_{*}(E,x):= \liminf_{r\rightarrow 0} \frac{\cH^{d}(E\cap B(x,r))}{(2r)^d}\]
and 
\[ \Theta^{d,*}(E,x):= \limsup_{r\rightarrow 0} \frac{\cH^{d}(E\cap B(x,r))}{(2r)^d}.
\]
The $d=1$ case is the Besicovitch $\frac{3}{4}$-Theorem \cite{Bes38}, which states that if $E\subseteq \mathbb{R}^{2}$ is a set of finite $1$-dimensional Hausdorff measure such that 
\begin{equation}
\label{e:3/4}
\Theta^{1}_{*}(E,x)>\frac{3}{4} \mbox{ for $\cH^{1}$-a.e. $x\in E$},
\end{equation}
then $E$ is $1$-rectifiable (and it is conjectured that $\frac{3}{4}$ can be replaced by $\frac{1}{2}$, see \cite{PT92,Far00,Far02} for some partial progress). The case for $d>1$ is due to Preiss \cite{Pr87} (which generalized earlier works of Mattila \cite{Mat75} and Marstrand \cite{Mar61}): there is a constant $\alpha(n,d)\in(0,1)$ such that for any $E\subseteq \mathbb{R}^{n}$ of locally finite $\cH^{d}$-measure, $E$ is $d$-rectifiable if
\begin{equation}
\label{e:preiss}
0< \alpha(n,d)\Theta^{d,*}(E,x)<\Theta^{d}_{*}(E,x)\mbox{ for $\cH^{d}$-a.e. $x\in E$}.
\end{equation}
In other words, rectifiability follows if the density of Hausdorff measure in a ball becomes roughly stable as the ball shrinks to a point at almost every point. This result requires information about the upper densities as well, but it gives a kind of generalization of Besicovitch's theorem using the following result \cite[2.10.19(5)]{Federer}: for any metric space $X$ of locally finite $d$-dimensional measure, 
\begin{equation}
\label{e:federer}
\Theta^{d,*}(X,\cdot)\leq 1 \;\; \mbox{ $\cH^{d}$-almost everywhere in $X$}
\end{equation} 
and in fact, this holds for spherical Hausdorff measure. In particular, this coupled with Preiss' result shows that the rectifiability of $E$ follows if 
\begin{equation}
\label{e:lowerdensity-preiss}
\alpha(n,d)< \Theta^{d}_{*}(E,x)\mbox{ for $\cH^{d}$-a.e. $x\in E$}.
\end{equation}
In fact, the same inequality is needed for Besicovitch's proof as well. 

The objective of our paper is to develop an analogue of these lower density criteria that guarantee a stronger rectifiable structure, in particular uniform rectifiability. A set $E\subseteq \R^{n}$ is said to be {\it $d$-uniformly rectifiable (UR)} if
\begin{enumerate}
\item it is {\it $C_{0}$-Ahlfors $d$-regular} for some $C_{0}>0$, meaning
\[
C_{0}^{-1}r^{d}\leq \cH^{d}(B(x,r)\cap E)\leq C_{0}r^{d} \mbox{ for all }x\in E, \;\; 0<r<\diam E,
\]
\item $E$ has {\it big pieces of Lipschitz images of $\R^{d}$ (BPLI)}, meaning there are $L,c>0$ so that for all $x\in E$ and $0<r<\diam E$, there is $f:\R^{d}\rightarrow \R^{n}$ $L$-Lipschitz so that $\cH^{d}(f(B(0,r))\cap B(x,r)\cap E)\geq cr^{d}$. 
\end{enumerate}

These sets were introduced by David and Semmes in \cite{DS}, the initial motivation being to characterize when certain singular integral operators were bounded on subsets of Euclidean space (see \cite{DS} for more discussion on this context). This began a program of trying to find various equivalent criteria for uniform rectifiability. We review a few such criteria here. Let $\cD^{E}$ denote the Christ-David cubes for $E$ (see Section \ref{s:proof of main} below). For each cube $Q\in \cD^E$, there is a ball $B_{Q}$ centered on and containing $Q$ of comparable size. Given  two closed sets $E$ and $F$, and $B$ a set, we denote
\[
d_{B}(E,F)=\frac{2}{\diam B}\max\left\{\sup_{y\in E\cap B}\dist(y,F), \sup_{y\in F\cap  B}\dist(y,E)\right\}.\]

For $C_{0}>0$, $\ve>0$, and $R\in \cD^{E}$, let
\begin{multline*}
{\rm BLWG}(C_0,\ve,R)=\sum\{(\diam Q)^{d}: Q\in \cD^{E},\;\; Q \subseteq R, \mbox{ and }   \\
 d_{C_0B_{Q}}(E,P)\geq \ve\mbox{ for all $d$-planes }P\}. 
\end{multline*}
We say $E$ satisfies the {\it bilateral weak geometric lemma} (BWGL) if for all $C_{0}>0$, $\ve>0$, and $R\in \cD^{E}$, we have ${\rm BLWG}(C_0,\ve,R)\lec \cH^{d}(R)$. Another condition is the {\it BAUP} ({\it bilateral approximation by unions of planes}) condition. We define ${\rm BAUP}(C_0,\ve,R)$ in a similar way to ${\rm BLWG}(C_0,\ve,R)$ except now we measure the distance between $E$ and any union of $d$-planes, not just one. 

It is shown that these two conditions are equivalent to UR (see \cite{of-and-on}). Part of the motivation for finding such criteria is that, depending on the kind of problem you are working on, it may be more natural to prove UR using one criterion over another. For example, the BWGL condition was crucial for showing uniform domains with UR boundaries are chord-arc domains \cite{AHMNT17}, and the BAUP condition was crucial to the solution of the David-Semmes Conjecture in codimension 1 (see \cite{NTV14}) and for studying the Dirichlet problem in domains with Ahlfors regular boundaries (see \cite{HLMN17}). 

Another motivation, which is also the motivation of this paper, is to revisit classical results from geometric measure theory and determine whether one can develop quantitative analogues. 

There is already a UR analogue of Preiss' result which was introduced in \cite{of-and-on}: for an Ahlfors $d$-regular set $E$ and $\ve>0$, let $\cA_{E}(c_1,\ve)$ be the set of pairs $(x,r)\in E\times (0,\diam E)$ for which there is a $c_1$-Ahlfors $d$-regular measure $\sigma_{x,r}$ with $\supp \sigma_{x,r}=E$ and 
\[
|\sigma_{x,r}(B(y,t))-t^{d}|<\ve r^{d} \mbox{ for all }y\in E\cap B(x,r).
\]
Let $\cB_{E}(c_{1},\ve)=E\times (0,\diam E)\backslash \cA_{E}(c_{1},\ve)$. We say satisfies the {\it weak constant density (WCD)} condition if there is $c_1$ so that $\cB_{E}(c_{1},\ve)$ is a Carleson set for every $\ve>0$, with norm depending on $\ve$. Recall, a set $A \subseteq E \times (0,\diam E)$ is a {\it Carleson set} if $\one_A d\cH^d(x)\frac{dr}{r}$ is a Carleson measure on $E \times (0,\diam E)$, and a measure $\mu$ is a {\it Carleson measure} on $E\times (0,\diam E)$  if there is $C>0$ so that $\mu(B(x,r)\times (0,r))\leq r^{d}$ for all $x\in E$ and $0<r$. This is certainly satisfied if the set of $(x,r)$ for which 
\[
|\cH^{d}|_{E}(B(y,t))-(2t)^{d}|<\ve r^{d} \mbox{ for all }y\in E\cap B(x,r) \mbox{ and }0<t<r.
\]
is a Carleson set for each $\ve>0$, which is a stronger condition than \eqref{e:preiss} in that the latter is implied by the former. David and Semmes first showed the WCD was satisfied by every UR set, and that it implied UR when $d=1,2,$ or $n-1$, and the general case was established by Tolsa \cite{Tol15}. 

There are similar results to these that characterize UR in terms of the fluctuations of the density of Hausdorff measure between scales rather than how far it is from $1$. For example, in \cite{CGLT16}, the authors show that an Ahlfors regular set $E$ is UR if $\Delta_{\mu}^{d}(x,r)^{2}\frac{dr}{r}d\cH^{d}(x)$ is a Carleson measure on $E\times (0,\diam E)$ where $\mu = \cH^{d}|_{E}$ and
\[
\Delta_{\mu}^{d}(x,r) = \left|\frac{\mu(B(x,r))}{(2r)^{d}}- \frac{\mu(B(x,r/2))}{r^d}\right|.
\]
In fact, their results hold more generally for Ahlfors regular {\it measures} $\mu$ and not just Hausdorff measure. See also \cite{TT15, Tol17} where this quantity is also used to characterize rectifiable sets and measures. \\

Our main result establishes a lower density criterion for uniform rectifiability using Hausdorff content $\cH^{d}_{\infty}$ rather than Hausdorff measure:

\begin{theorem}\label{t:WLD}
Let $E \subseteq \R^n$ be a $C_0$-Ahlfors $d$-regular set. For $\ve >0,$ let $\cB_\emph{WLD}(\ve)$ be the set of $(x,r) \in E \times (0,\infty)$ for which there exists $y \in E \cap B(x,r)$ and $0 < t <  r$ such that
\begin{equation}
\label{e:WLD}
\cH_\infty^d(E \cap B(y,t)) < (2t)^d-\ve (2r)^d
\end{equation}
If $E$ satisfies the \textit{weak lower density condition} \emph{(WLD)}, meaning $\cB_\emph{WLD}(\ve)$ is a Carleson set for each $\ve >0,$ then $E$ is UR. 
\end{theorem}
In other words, if we have nice estimates on how often the density of Hausdorff {\it content} dips below $1$, then we can guarantee UR. We explain later why we require Hausdorff content rather than Hausdorff measure. Notice also that if $(x,r)\not\in \cB_{\rm WLD}(\ve)$, this means \eqref{e:WLD} fails for all balls $B(y,t)$ with $y\in E\cap B(x,r)$ and $0<t<r$, but this doesn't say the density of Hausdorff content is not much smaller than $1$ in {\it all} balls , since \eqref{e:WLD} fails trivially for all $t<\epsilon^\frac{1}{d}r$. Hence, $(x,r)\not\in \cB_{\rm WLD}(\ve)$ only gives information about the densities of Hausdorff content in balls that aren't too much smaller than $r$. 

Note that while we do rely on work from \cite{Tol15} in our proof, there is still much work to do. For one, our condition is in terms of Hausdorff {\it content} and not Hausdorff {\it measure}, but more importantly our condition does not require information about how often the density of Hausdorff content dips {\it above} 1, whereas the conditions of David, Semmes, and Tolsa ask that the density of Hausdorff measure is not too much above or below $1$ in most balls. 

The converse to Theorem \ref{t:WLD} is not true. If such a result did hold, we could find a Carleson condition on $\mathscr{B}_{\text{WLD}}(\ve)$ for any UR set, with norm depending only on $\ve$ and the constants appearing in the definition of UR. We show this is not possible by constructing a sequence of UR sets $E_n$ with uniform UR constants, so that the associated sequence of Carleson norms blows up. 

For $n \in \bN,$ let $E_n \subseteq [0,1] \subseteq \R$ be the set consisting every other dyadic interval of length $2^{-n}$ contained in $[0,1],$ that is, 
\[E_n = \bigcup_{k=1}^{2^{n-1}} [(2k-1)2^{-n},(2k)2^{-n}]. \]  
It is not difficult to show that $E_n$ is UR with constants $C_0 = 4, \ L = 1$ and $c = \tfrac{1}{4}$ for each $n \in \bN.$ Let $\ve < \tfrac{1}{3}.$ Since each ball $B$ (where balls in this setting are actually intervals) centered on $E_n$ with radius larger than $2^{-n}$ satisfies
\[ \cH^1_\infty(E_n \cap B) \leq \frac{2}{3}(2r_B) < (1-\ve)(2r_B), \]
it follows that $E_n \times (2^{-n} , 1) \subseteq \cB_{\text{WLD}}(\ve).$ Then for any $x \in E_n$ and $2^{-n} < R< 1,$
\begin{align*}
\int_{2^{-n}}^R \int_{B(x,R)} \one_{\cB_\text{WLD}(\ve)}(y,r) \, d\cH^1|_{E_n}(y) \frac{dr}{r} &\geq \log( R2^n)\cH^1(E_n \cap B(x,R)) \\
&\gtrsim \log(R2^n)R
\end{align*}
and the right-hand side goes to $\infty$ as $n \rightarrow \infty$.\\

One reason we only need to control how often the density of Hausdorff content dips below $1$ is the following result that may be of independent interest, which says that actually content doesn't jump above $1$ too much anyway, and this holds for quite general sets, not just Ahlfors regular sets. One can view this as a quantitative version of \eqref{e:federer}. 

\begin{theorem}
\label{t:main}
Let $X$ be a compact $C$-doubling metric space, meaning that every ball $B$ in $X$ can be covered by at most $C$ many balls of half the radius. Let $\ve,\rho\in (0,1)$, $A>0$, and $X_{n}$ be a sequence of maximal $2^{-n}$-separated sets in $X$, and $\cB=\{B(x,2^{-n}):x\in X_{n},n\in \bN\}$, then 
\[
\sum \left\{r_{B}^{s}: B\in \cB, \frac{\cH^{s}_{\rho r_{B}}(X\cap AB)}{(2r_{B})^{s}}>1+\ve\right\} \lec_{C,A} \frac{\log \frac{1}{\min\{\rho,\ve/s\}}}{\ve} \cH^{s}(X).
\]
\end{theorem}

Recall that a metric space is $C$-{\it doubling} if any ball $B\subseteq X$ may be covered by at most $C$-many balls of half its radius.

The proof of this theorem is mostly an adaptation of the geometric martingale techniques in Schul's proof of the Analyst's Traveling Salesman Theorem in Hilbert space \cite{Sch07-TST}. In that paper, Schul needs to control the sum of diameters of balls centered along a curve $\Gamma$ of finite length for which the portion of $\Gamma$ in these balls consists of more than one approximately straight curve segments. In such balls, $\Gamma$ will have large Hausdorff content, and it is really that property that he is using implicitly in his proof, so his method can be extrapolated to sets other than curves, or even sets of non-integer dimension.\\

The referee asked some very interesting questions that we were not able to answer but are venues for future work. 

First, in light of the example $E_n$ we constructed, the referee asked us whether the assumptions in Theorem \ref{t:WLD} were so strong that they may imply a stronger property than being UR: having big pieces of Lipschitz graphs. 

We say a set $E$ has {\it big pieces of Lipschitz graphs} (or BPLG) if there are $L,c>0$ so that for all $x\in E$ and $0<r<\diam E$, there is a $d$-dimensional $L$-Lipschitz graph $\Gamma\subseteq \R^{n}$ (that is, a rotated copy of the graph of an $L$-Lipschitz function $f:\R^{d}\rightarrow \R^{n-d}$) so that $\cH^{d}(\Gamma\cap  B(x,r)\cap E)\geq cr^{d}$. The distinction between BPBI and BPLG might seem arbitrary, but BPLG is quite crucial in some applications, see for example \cite{DJ90}. 

It is part of the lore in the theory of UR that this property implies but is not equivalent to UR due to an unpublished example of Hrycak\footnote{The first author learned this from John Garnett who had emailed Steve Hofmann who had emailed Stephen Semmes who learned it from Hrycak.}. The example is actually a special case of the classical Venetian blinds construction (see \cite{Falconer,Fal86}), but it was Hrycak's idea to use it to show BPBI $\not\Rightarrow$ BPLG. A start towards asking the referee's question would be to see if this example satisfies the WLD condition. The authors believe they can show  (not reported here) that it satisfies a weaker WLD type condition--the same condition but with Hausdorff measure instead of Hausdorff content--but the proof relies on the additivity of Hausdorff measure and it is not clear whether it can be extended to the case of Hausdorff content. 

The second question the referee asked was whether we could generalize the result to sets that are not Ahlfors regular. The definition of UR is no longer appropriate in this setting, but there are ways of generalizing results from UR to more general settings: In \cite{AV19}, the first author and Villa generalized many results from UR to {\it lower content regular sets}, which are sets where we assume $\cH^{d}_{\infty}(E \cap B(x,r))\gec r^{d}$ for all $x\in E$ and $0<r<\diam E$. It turns out in this setting that the geometric sums ${\rm BWGL}$ and ${\rm BAUP}$ are still meaningful in this context in the sense that the following estimate holds inside any cube $R$:
\begin{equation}
\label{e:bwglbaup}
\cH^{d}(R)+{\rm BWGL}(C_0,\ve,R)\sim \cH^{d}(R)+{\rm BAUP}(C_0,\ve,R)
\end{equation}
and in fact these are comparable to other sums like ${\rm BWGL}$ that appear in the theory of UR, see \cite{AV19} for more details. 

The most natural way to define a quantity ${\rm WLD}(R)$, like ${\rm BWGL}(C_0,\ve,R),$ is to let it equal the sum of $(\diam Q)^{d}$ for cubes $Q$ in $R$ for which \eqref{e:WLD} holds for $CB_Q$ for some $C>0$. The referee's question is whether $\cH^{d}(R) + {\rm WLD}(R)$ has any relation to $\cH^{d}(R)+{\rm BWGL}(C_0,\ve,R)$. The earlier example we constructed shows the two are not comparable: if $R=E_n$, then $\cH^{d}(R)+{\rm BWGL}(C_0,\ve,R)\sim 1$ whereas $\cH^{d}(R) + {\rm WLD}(R)\rightarrow\infty$ as $n\rightarrow\infty$. It could be that we still have $\cH^{d}(R)+{\rm BLWG}(C_0,\ve,R)\lec \cH^{d}(R) + {\rm WLD}(R)$. Some of our arguments take us part of the way, however our work below takes advantage of the fact that Ahlfors regular sets form a compact family in the sense that if we have a sequence of such sets containing the origin, then we can pass to a subsequence so that they converge to another Ahlfors regular set, and in particular we take advantage of Hausdorff measure being locally finite on this set, whereas a sequence of $d$-lower regular sets of locally finite $d$-measure may not converge to a set of locally finite $d$-measure.

\subsection*{Acknowledgments} We would like to thank the anonymous referees for spotting several mistakes and for their many useful comments and suggestions that greatly improved the paper. The second author was supported by The Maxwell Institute Graduate School in Analysis and its Applications, a Centre for Doctoral Training funded by the UK Engineering and Physical Sciences Research Council (grant EP/L016508/01), the Scottish Funding Council, Heriot-Watt University and the University of Edinburgh.

\section{Notation}

We will write $a\lesssim b$ if there is $C>0$ such that $a\leq Cb$ and $a\lesssim_{t} b$ if the constant $C$ depends on the parameter $t$. We also write $a\sim b$ to mean $a\lesssim b\lesssim a$ and define $a\sim_{t}b$ similarly.

Let $X$ be a metric space. We will denote the distance between two points $x,y\in X$ by $|x-y|$. For sets $A,B\subset X$, let 
\[\dist(A,B)=\inf\{|x-y| \; |\; x\in A,y\in B\}, \;\; \dist(x,A)=\dist(\{x\},A),\]
and 
\[\diam A=\sup\{|x-y|\;| \; x,y\in A\}.\]

For $x\in X$ and $r>0$, we will let $B(x,r)$ be the closed ball centered at $x$ of radius $r$. If $B=B(x,r)$ and $\lambda >0$, we will let $\lambda B= B(x,\lambda r)$. For a closed ball $B$, we let $B^\circ$ be the open ball with the same centre and radius as $B.$  

We recall the definition of Hausdorff measures and contents, but more information can be found in \cite{Mattila}: for $A\subseteq X$, $s\geq 0$, and $\delta>0$, we define
\[
\cH^{s}_{\delta}(A)=\inf\left\{\sum (\diam A_i)^{s}: A\subseteq\bigcup A_{i},\;\; \diam A_{i}\leq \delta\right\}. 
\]
The $s$-dimensional Hausdorff content is defined to be $\cH^{s}_{\infty}(A)$, and $s$-dimensional Hausdorff measure is defined to be the limit
\[
\cH^{s}(A)=\lim_{\delta\rightarrow 0} \cH^{s}_{\delta}(A).
\]
Notice that $\cH^{s}_{\delta}(A)$ is decreasing in $\delta$, that is,

\begin{equation}
\label{e:Hdecreasing}
\cH^{s}_{\delta}(A)\leq \cH^{s}_{\delta'}(A) \leq \cH^{s}(A) \;\; \mbox{ for }\delta'\leq \delta.
\end{equation}

\section{Weak convergence of measures}\label{s:WCM}

In this section we consider the weak convergence of a sequence of measures of the form $\mu_k = \cH^d_{\rho_k}|_{E_k},$ where $\rho_k \rightarrow 0.$ In what follows, unless stated otherwise, a measure will simply refer to a monotonic, countably subadditive set function which vanishes for the empty set. In particular, we do not require a measure to be additive. The results of this section will be used in the proof of Theorem \ref{t:WLD}, we delay their proofs until the appendix. 

For a measure $\mu$ and a function $f: \R^n \rightarrow [0,\infty)$, define the Choquet integral of $f$ with respect to $\mu$ by the formula
\[
\int f \, d\mu = \int_0^\infty \mu( \{x \in \R^n : f(x) > t \}) \, dt.
\]
For a real valued function $f: \R^n \rightarrow \R,$ let $f^+ = \max\{f,0\}$ and $f^{-} = \max \{-f,0 \}.$ Define the Choquet integral of $f$ with respect to $\mu$ by 
\[
\int f \, d\mu = \int f^+ \, d\mu - \int f^{-} \, d\mu. 
\]

\begin{definition}
	Let $\{\mu_k\}$ be a sequence of measures on $\R^n.$ We say the sequence $\{\mu_k\}$ converges \textit{weakly} to a Radon measure $\mu$, and write 
	\[
	\mu_k \rightharpoonup \mu,
	\]
	if
	\[
	\lim_{k \rightarrow \infty} \int \vp \, d\mu_k = \int \vp \, d\mu \quad \text{for all} \ \vp \in C_0(\R^n). 
	\]
	Here, $C_0(\R^n)$ is the space of continuous functions of compact support. 
\end{definition}

We state some general results about the weak convergence of measures. The results are essentially those found in Chapter 1 of \cite{Mattila} and Chapter III.5 of \cite{of-and-on}.

\begin{lemma}\label{l:semicontm}
	Suppose $\{\mu_k \}$ is a sequence of measures converging weakly to a Radon measure $\mu.$ For $K \subseteq \R^n$ compact and $U \subseteq \R^n$ open we have 
	\[
	\mu(K) \geq \limsup_{k \rightarrow \infty} \mu_k(K)
	\]
	and
	\[
	\mu(U) \leq \liminf_{k \rightarrow \infty} \mu_k(U). 
	\]
\end{lemma}

\begin{lemma}\label{l:convergingm}
	Suppose $\{\mu_k\}$ is a sequence of measures converging weakly to a Radon measure $\mu$. Suppose additionally there exists $C_0 >0$ such that each $\mu_k$ is $C_0$-Ahlfors $d$-regular (in the sense that it satisfies the upper and lower regularity condition with constant $C_0$, but may not be additive). Then, for any ball $B$, we have 
	\[
	\lim_{k \rightarrow \infty} \left( \sup_{p \in B \cap \emph{supp} \, \mu} \emph{dist}(p, \emph{supp} \, \mu_k) \right) = 0
	\]
	and
	\[
	\lim_{k \rightarrow \infty} \left(\sup_{p \in B \cap \emph{supp} \, \mu_k} \emph{dist}(p, \emph{supp} \, \mu) \right) = 0.
	\]
\end{lemma}

The main result of this section is the following. 

\begin{lemma}\label{l:subseqm}
	Let $\{E_k\}$ be a sequence of $C_0$-Ahlfors $d$-regular sets in $\R^n$ and $\{\rho_k\}$ a sequence of positive real numbers such that $\rho_k \rightarrow 0.$ Let $\mu_k = \cH^d_{\rho_k}|_{E_k},$ then there exists sub-sequence $\{\mu_{k_j}\}$ and a Radon measure $\mu$ such that $\mu_{k_j} \rightharpoonup \mu.$
\end{lemma}

\section{Proof of Theorem \ref{t:WLD}}
\label{s:proof of main}

We recall the properties of the Christ-David cubes from \cite{Dav88,Chr90}. Let $E \subseteq \R^n$ be $C_0$-Ahlfors $d$-regular. Let $X_n$ be a sequence of maximal $2^{-n}$-separated nets in $E$ and
\[
\cD^E = \bigcup_{j \in \bZ} \cD^E_j
\]
denote the Christ-David cubes with respect to this sequence of nets. If the context is clear, we shall drop the superscript $E$. For a measure $\mu,$ denote $\cD^\mu = \cD^{\text{supp} \, \mu}.$ The cubes in $\cD$ satisfy the following: 
\begin{enumerate}[(i)]
\item For each $j \in \bZ,$ $E = \bigcup_{Q \in \cD_j} Q$.
\item If $Q \in \cD_j$ and $Q^\prime \in \cD_k$ for $j \leq k$ then either $Q^\prime \subseteq Q$ or $Q \cap Q^\prime = \emptyset.$
\item There exists $c_0$ such that the following holds. For $j \in \bZ$ and $Q \in \cD_j$, let $\ell(Q) = 2^{-j},$ there is $x_Q \in Q$ such that 
\[
B_E(x_Q,c_0\ell(Q)) \subseteq Q \subseteq B_E(x_Q , \ell(Q)). 
\]

\end{enumerate}
Given a cube $Q$, denote
\[
B_Q = B(x_Q,r_Q) = B(x_Q,3\ell(Q)).
\]
We say a collection of cubes $\mathscr{C} \subseteq \cD$ satisfies a \textit{Carleson packing condition} if there exists $C>0$ such that for each $R \in \cD,$
\[
\sum_{\substack{Q \in \mathscr{C} \\ Q \subseteq R}} \ell(Q)^d  \leq C\ell(R)^d.
\]
The main idea behind the proof of Theorem \ref{t:WLD} is that if $E$ satisfies the WLD condition, then at most scales and locations, $E$ may be approximated by the support of some uniform measure (see definition below). We use this, along with \cite{Tol15}, to finish the proof. 

We recall some notation and results from \cite{Tol15}. A Borel measure $\mu$ in $\R^n$ is said to be \textit{$d$-uniform} if the exists a constant $c>0$ such that 
\[
\mu(B(x,r)) = cr^d
\]
for all $x \in \supp \mu$ and $r > 0.$

Given a ball $B$ and two Radon measures $\mu$ and $\nu$ such that $\supp \mu \cap B \not= \emptyset$ and $\supp \nu \cap B \not=\emptyset$, define
\[
d_B(\mu,\nu) = \sup_{x \in B \cap \text{supp} \, \nu} \text{dist}(x,\text{supp} \, \mu) + \sup_{x \in B \cap \text{supp} \, \mu} \text{dist}(x,\text{supp} \, \nu).
\]
For a Radon measure $\mu$ and a constant $\eta >0,$ let $\mathscr{N}_0(\mu,\eta)$ be the collection of balls $B$ such that there exists a $d$-uniform measure $\nu$ in $\R^n$ satisfying
\[
d_B( \mu, \nu) \leq \eta.
\]
Furthermore, let $\mathscr{N}(\mu,\eta)$ denote the set of cubes $Q  \in \cD^\mu$ such that $B_Q \in \mathscr{N}_0(\mu,\eta).$ When the context is clear, we shall simply write $\mathscr{N}(\eta)$ and $\mathscr{N}_0(\eta).$ 

Although not explicitly stated, in the Section 4 of \cite{Tol15}, Tolsa proves the following:
\begin{proposition}\label{p:Tolsa}
Suppose $\mu$ is an Ahlfors $d$-regular measure and $\cD^\mu \setminus \cN(\eta)$ satisfies a Carleson packing condition for each $\eta >0.$ Then, $\mu$ is UR. 
\end{proposition}

Given the results on uniform measure contained in Sections 1 - 3 of \cite{Tol15}, the proof of the above Proposition \ref{p:Tolsa} is contained within the proof of Theorem 1.1 of the aforementioned paper, beginning on page 16. With the following result of David and Semmes (see \cite[Chapter III.5]{of-and-on}), Proposition \ref{p:Tolsa}  proves the WCD condition implies UR. 

\begin{proposition}
Suppose $\mu$ satisfies the \emph{WCD} condition, then $\cD^\mu \setminus \cN(\eta)$ satisfies a Carleson packing condition for each $\eta >0.$
\end{proposition}

The main goal of this section is to prove the following lemma, analogous to the above result of David and Semmes. This, along with Proposition \ref{p:Tolsa}, will finish the proof of Theorem \ref{t:WLD}.

\begin{lemma}\label{l:uniform}
Suppose $E \subseteq \R^n$ satisfies the WLD condition. Then $\cD^E\setminus\mathscr{N}(\eta)$ satisfies a Carleson packing condition for each $\eta >0.$
\end{lemma}

For $A \geq 1$ and $\ve,\rho > 0,$ let $\cG(A,\ve,\rho)$ be the collection of cubes $Q \in \cD$ such that
\[
\cH^d_{\rho r_Q} (E \cap AB_Q) \leq (1+\ve)(2Ar_Q)^d
\]
and 
\[
\cH_\infty^d(E \cap B(x,r)) \geq (2r)^d - \ve (2Ar_Q)^d
\]
for all $x \in E \cap AB_Q$ and $0 < r< Ar_Q.$

We prove Lemma \ref{l:uniform} by showing, for suitable choices of $A,\ve, \rho$, that for each $Q \in \cG(A,\ve,\rho)$ there is a $d$-uniform measure such that $E$ is locally well-approximated by $\text{supp} \, \mu.$ The Carleson packing condition on $\cD\setminus \mathscr{N}(\eta)$ will follow from packing conditions on $\cB(A,\ve,\rho) = \cD\setminus \cG(A,\ve,\rho),$ which in turn follow from Theorem \ref{t:main} and the definition of WLD.

Denote by $\cB_1(A,\ve)$ the set of cubes $Q$ in $\cD$ for which there exists $y \in E \cap AB_Q$ and $0 < r < Ar_Q$ satisfying 
\begin{align}\label{e:lower}
\cH^d_\infty(E \cap B(y,r)) < (2r)^d - \ve (2Ar_Q)^d.
\end{align}
Additionally, let $\cB_2(A,\ve,\rho)$ denote the set of cubes $Q$ such that
\[
\cH^d_{\rho r_Q}(E \cap AB_Q) > (1+\ve) (2Ar_Q)^d.
\]

\begin{lemma}\label{l:car1}
Suppose $E$ satisfies the WLD condition, then $\cB_1(A,\ve)$ satisfies a Carleson packing condition for each $A \geq 1$ and $\ve >0.$ 
\end{lemma}

\begin{proof}
Let $A \geq 1,$ $\ve >0$ and $R \in \cD.$ Let $Q \in \cB_1(A,\ve)$, and let $B = B(y,r)$ be the ball satisfying \eqref{e:lower}. Thus, if $x \in Q$ then $y \in B(x,2Ar_Q)$ and 
\[
\cH^d_\infty(E \cap B(y,r)) < (2r)^d - \ve (2Ar_Q)^d = (2r)^d - 2^{-d}\ve (4Ar_Q)^d,
\]
that is $(x,2Ar_Q) \in \cB_\text{WLD}(\ve/2^d) \subseteq \cB_\text{WLD}(\ve/4^d).$ Similarly, one can show that $(x,\alpha A r_Q) \in \cB_\text{WLD}(\ve/4^d)$ for each $2 \leq \alpha \leq 4.$ Let $k^*$ be such that $R \in \cD_{k^*}.$ Denoting $\cB_{k,1}(A,\ve) = \cB_1(A,\ve) \cap \mathscr{D}_k,$ we have
\begin{align*}
\sum_{\substack{Q \in \cB_1(A,\ve) \\ Q \subseteq R}} \ell(Q)^d &\leq \sum_{k = k^*}^\infty  \int_{3A2^{-k+1}}^{3A2^{-k+2}} \sum_{\substack{Q \in \cB_{k,1}(A,\ve) \\ Q \subseteq R}}\ell(Q)^d \, \frac{dr}{r} \\
&\hspace{-4em}\lesssim \sum_{k = k^*}^\infty  \int_{3A2^{-k+1}}^{3A2^{-k+2}} \sum_{\substack{Q \in \cB_{k,1}(A,\ve) \\ Q \subseteq R}} \cH^d ( \{ x \in Q : (x,r) \in \cB_\text{WLD}(\ve/4^d)\}) \, \frac{dr}{r} \\ 
&\hspace{-4em}\lesssim \int_{Ar_R}^{4Ar_R} \int_{AB_R} \one_{\cB_\text{WLD}(\ve/4^d)}(x,r) \, d\cH^d|_E(x) \frac{dr}{r} \\
&+  \int_0^{Ar_R} \int_{AB_R} \one_{\cB_\text{WLD}(\ve/4^d)}(x,r) \, d\cH^d|_E(x) \frac{dr}{r} \\
&\hspace{-4em}\lesssim_{A,\ve} \ell(R)^d,
\end{align*}
where the second inequality follows from Ahlfors regularity and the final inequality follows from Ahlfors regularity and the fact that $\cB_\text{WLD}(\ve)$ is a Carleson set. 
\end{proof}

\begin{lemma}\label{l:car2}
The set $\cB_2(A,\ve,\rho)$ satisfies a Carleson packing condition for each $A \geq 1$ and $\ve,\rho >0$. 
\end{lemma}

\begin{proof}
This is an immediate consequence of Theorem \ref{t:main}. 
\end{proof}

Notice $\cB(A,\ve,\rho) \subseteq \cB_1(A,\ve) \cup \cB_2(A,\ve,\rho).$ Thus, combining Lemma \ref{l:car1} and Lemma \ref{l:car2}, it follows that $\cB(A,\ve,\rho)$ also satisfies a Carleson packing condition for each $A \geq 1$ and $\ve,\rho >0.$ To finish the proof of Lemma \ref{l:uniform}, it now remains to show this implies a Carleson packing condition on $\cD\setminus\cN(\eta).$

\subsection{Approximation by uniform measures} In this section we prove that, for a suitable choice of $A \geq 1$ and $\ve,\rho >0,$ if $Q \in \cG(A,\ve,\rho),$ then there exists a $d$-uniform measure $\mu_Q$ which well-approximates $Q$. We first consider a related collection of sets. 

\begin{definition}
Let $\mathscr{U}(A,C_0,\ve,\rho)$ be the collection of subsets $E \subseteq \R^n$ which are $C_0$-Ahlfors $d$-regular, contain the origin, and satisfy:
\begin{enumerate}
\item $\cH_{\rho r_\bB}^d(E \cap A\bB) \leq (1+\ve)(2A)^d,$
\item $\cH_{\infty}^d(E \cap B) \geq (2r_B)^d -\ve(2A)^d$ for all $B$ centered on $E \cap A\bB$ with $r_B \leq A.$ 
\end{enumerate}
Here, $\bB$ denotes the unit ball in $\R^n$ centered at the origin.
\end{definition}

Most of the details of the following lemma are contained in the proof of \cite[Lemma III.5.13]{of-and-on}, we include a proof for the reader's convenience. 

\begin{lemma}
Let $\eta >0$ be given. There is $A \geq 1$ such that if $\mu$ is a $C_0$-Ahlfors $d$-regular Radon measure satisfying
\[
\mu(B) = (2r_B)^d
\]
for all $B$ centered on \emph{supp}$\, \mu \cap A\bB$ with $r_B \leq A,$ 
then there is a $d$-uniform measure $\nu$ such that
\[
d_\bB(\mu,\nu) \leq \eta.
\]
\end{lemma}

\begin{proof}
Suppose the statement is false. We can find a sequence of real numbers $A_j \rightarrow \infty$ and $C_0$-Ahlfors $d$-regular Radon measures $\mu_j$ such that $\mu_j(B) = (2r_B)^d$ for all $B$ centered on $\text{supp}\,\mu_j \cap A_j\bB$ with $r_B \leq A_j$, but $d_\bB(\mu_j,\nu) > \eta$ for all $d$-uniform measures $\nu.$ By extracting a subsequence if necessary, we can assume $\mu_j \rightharpoonup \mu,$ where $\mu$ is a Radon measure. This is possible by \cite[Lemma 1.23]{Mattila} since the $\mu_j$ are $C_0$-Ahlfors $d$-regular. 

We claim $\mu$ is $d$-uniform. Let $B = B(x_B,r_B)$ be centered on $\text{supp}\,\mu.$ For each $j$, let 
\[ \delta_j = \sup_{p \in B \cap \supp \mu} \dist(p,\supp \mu_j) \] 
so that there exists $x_j \in \supp \mu_j$ satisfying $|x_B - x_j| \leq \delta_j.$ Let $B_j = B(x_j,r_B + \delta_j)$. Clearly $B \subseteq B_j$, and $r_{B_j} \rightarrow r_B$ by Lemma \ref{l:convergingm}. Let $\ve > 0$ be small. For $j$ large enough $(1+\ve)r_{B_j} \leq A_j,$ so by Lemma \ref{l:semicontm}, 
\begin{align*}
\mu(B) &\leq \mu((1+\ve)B^\circ) \leq \liminf_{j \rightarrow \infty} \mu_j((1+\ve)B^\circ) \leq \liminf_{j \rightarrow \infty} \mu_j((1+\ve)B_j) \\
&= \liminf_{j \rightarrow \infty} (2(1+\ve)r_{B_j})^d = (2(1+\ve)r_B)^d. 
\end{align*}
Since $\ve > 0$ was arbitrary, we conclude that 
\[ \mu(B) \leq (2r_B)^d. \] 
Similarly, let $B^\prime_j = B(x_j,r_B - \delta_j).$ Then, $B_j^\prime \subseteq B,$ $r_{B_j'} \rightarrow r_B$, and $r_{B_j'} \leq A_j$ for $j$ large enough. Hence 
\[
\mu(B) \geq \limsup_{j \rightarrow \infty} \mu_j(B) \geq \limsup_{j \rightarrow \infty} \mu_j(B^\prime_j) \geq \limsup_{j \rightarrow \infty} (2r_{B^\prime_j})^d = (2r_B)^d.
\]
For all $j$ large enough $d_\bB(\mu_j,\mu) \leq \eta$, by Lemma \ref{l:convergingm}. This contradicts the assumptions on the $\mu_j$ since $\mu$ is $d$-uniform.
\end{proof}

\begin{lemma}
Let $A \geq 1$ and $\eta>0$ be given. There exists $\ve, \rho >0$ so that if $E \in \mathscr{U}(2A,C_0,\ve,\rho)$ then there is a Radon measure $\mu$ such that 
\[
d_\bB(\cH^d|_E,\mu) \leq \eta
\]
and
\[
\mu(B) = (2r_B)^d
\]
for all $B$ centered on \emph{supp}$\, \mu \cap A\bB$ such that $r_B \leq A.$ 
\end{lemma}

\begin{proof}
Suppose the lemma is false. Then, there exists a sequence of sets $E_j$ and real numbers $\ve_j,\rho_j \rightarrow 0$ such that $E_j \in \mathscr{U}(2A,C_0,\ve_j,\rho_j)$ but the conclusion of the above lemma is false for each $j$. Let $\mu_j =\cH^d_{\rho_j}|_{E_j}.$ By Lemma \ref{l:subseqm}, we can extract a subsequence (which we do not relabel) such that $\mu_j \rightharpoonup \mu$ where $\mu$ is a Radon measure. Note that
\[
\mu(2A\bB^\circ) \leq \liminf_{j \rightarrow \infty} \mu_j (2A\bB^\circ) \leq \liminf_{j \rightarrow \infty} (1+\ve_j)(4A)^d = (4A)^d.
\]
Let $B$ be a ball centered on $\text{supp} \, \mu \cap A\bB$ with $r_B \leq A.$ As in the proof of the previous lemma, for each $j$ we can find balls $B_j$ centered on $\text{supp} \, \mu_j \cap A\bB$ such that $B_j \subseteq B$ and $r_{B_j} \rightarrow r_B.$ Then  
\begin{align*}
\mu(B) &\geq \limsup_{j \rightarrow \infty} \mu_j(B) \geq \limsup_{j \rightarrow \infty} \mu_j(B_j) \geq \limsup_{j \rightarrow \infty} \cH^d_\infty|_{E_j}(B_j) \\
&\geq \limsup_{j \rightarrow \infty} \left((2r_{B_j})^d-\ve_j (4A)^d\right) = (2r_B)^d. 
\end{align*}
We claim, in fact, $\mu(B) = (2r_B)^d.$ Assume $\mu(B) > (2r_B)^d$. For each $x \in \supp \mu \cap 2A\bB^\circ,$ let
\[ r_x = \sup \{r : B(x,r) \subseteq 2A\bB^\circ \ \text{and} \ B(x,r) \cap B =\emptyset \}. \]
Then, let 
\[\cB_x = \{B(x,r) : 0 < r < r_x \} \quad \text{and} \quad \cB' = \bigcup_{x \in \supp \mu \cap 2A\bB^o} \cB_x.\]
Notice each ball $B' \in \cB'$ is contained in $2A\bB^\circ$ and has empty intersection with $B.$ By the Vitali Covering Theorem (\cite[Theorem 2.8]{Mattila}), we may find a disjoint collection of balls $\cB \subseteq \cB'$ such that
\[ \cH^d|_{\supp \mu}\left( 2A\bB^\circ \setminus  \left( B \cup \bigcup_{B' \in \cB} B' \right) \right) = 0,\]
in particular
\[ \cH^d_\infty\left( \supp \mu \cap 2A\bB^\circ \setminus  \left( B \cup \bigcup_{B' \in \cB} B' \right) \right) = 0,\]
Let $\tau > 0$ be so that $\mu(B) = \tau + (2r_B)^d.$ By \cite[Theorem 2.1]{MM97}, $\cH^d_\infty$ is an upper semicontinuous function when acting on compact subsets of a compact metric space equipped with the Hausdorff norm. Using this, with the fact that the balls in $\cB$ are pairwise disjoint and have empty intersection with $B$, for any $0 < \alpha < 1$ we get
\begin{align*}
(4A)^d &\geq \mu(2A\bB^\circ)  \geq \mu(B) + \sum_{B^\prime \in \cB} \mu(B^\prime) \geq \tau + (2r_B)^d + \sum_{B^\prime \in \cB} (2r_{B^\prime})^d \\
&\geq \tau + \cH^d_\infty\left( \supp \mu \cap \left( B \cup \bigcup_{B' \in \cB'} B'\right
)\right) \\
&\geq \tau +  \cH^d_\infty (\supp \mu \cap 2A\bB^\circ) \\
&\hspace{4em} - \cH^d_\infty\left( \supp \mu \cap 2A\bB^\circ \setminus \left( B \cup \bigcup_{B' \in \cB} B' \right) \right) \\
&\geq \tau + \cH^d_\infty(\supp \mu \cap 2A\alpha\bB) \\
&\geq \tau + \limsup_{j \rightarrow \infty} \cH^d_\infty(E_j \cap 2A\alpha\bB) \geq \tau +  (4A\alpha)^d.
\end{align*}
Taking $\alpha \rightarrow 1$ gives $(4A)^d \geq \tau +  (4A)^d$ which is a contradiction and proves the claim.

We finish the proof of the lemma by noting that for $j$ large enough,
\[
d_\bB (\cH^d|_{E_j}, \mu) =d_\bB( \mu_j, \mu) \leq \eta
\]
by Lemma \ref{l:convergingm}, which is a contradiction. 
\end{proof}

Combining the above two lemmas, for $\eta > 0$ we can find $A \geq 1$ and $\ve,\rho > 0$ (depending on $\eta$) so that for any $E \in \mathscr{U}(2A,C_0,\ve,\rho),$ there exists a $d$-uniform measure $\mu$ satisfying $d_{\bB}(\cH^d|_E,  \mu) \leq \eta.$ By re-scaling and translation, we have the following.

\begin{lemma}\label{l:subset}
Let $\eta >0.$ There exist $A \geq 1$ and $\ve, \rho >0$ so that for any $Q \in \cG(A,\ve,\rho),$ there is a $d$-uniform measure $\mu$ such that 
\[
d_{B_Q}(\cH^d|_E, \mu) \leq \eta.
\]
\end{lemma}

\begin{proof}[Proof of Lemma \ref{l:uniform}]

Let $\eta >0.$ By Lemma \ref{l:subset}, we can find $A \geq 1$ and $\ve,\rho>0$ dependent on $\eta$ such that $\cG(A,\ve,\rho) \subseteq \mathscr{N}(\eta).$ Hence $\cD\setminus \mathscr{N}(\eta) \subseteq \cB(A,\ve,\rho).$ The proof of Lemma \ref{l:uniform} is completed by noting that 
\[
\sum_{\substack{Q \in \cD \setminus \mathscr{N}(\eta) \\ Q \subseteq R}} \ell(Q)^d \leq \sum_{\substack{Q \in \cB(A,\ve,\rho) \\ Q \subseteq R}} \ell(Q)^d \lesssim_\eta \ell(R)^d,
\]
where the last inequality follows since $\cB(A,\ve,\rho)$ is a Carleson set. 
\end{proof}

\section{Proof of Theorem \ref{t:main}}

\subsection{Notation and conventions}

In the sections below, $X$ will denote a $C$-doubling metric space. By the Kuratowski embedding theorem, $X$ isometrically embeds into $\ell^{\infty}(X)$, so without loss of generality, we will assume $X$ is a subset of some Banach space $\mathscr{X}$. Thus, whenever we talk about a ball $B(x,r)$, we mean the closed ball centered at $x$ of radius $r$ {\it with respect to $\mathscr{X}$}. In this way, the ball in $X$ is just $B_{X}(x,r)=X\cap B(x,r)$. The diameter of a set is defined in the usual way, but note that, while for a metric space $X$ we could have $\diam B_{X}(x,r)=0$, we always have $\diam B(x,r)=2r$.

We will also denote
\[
\mu = \cH^{s}|_{X}.
\]

\subsection{Cubes}

Before embarking on the proof of Theorem \ref{t:main}, we need to recall Schul's cubes \cite{Sch07-TST}. These are a family of subsets of $X$, that have properties similar to dyadic cubes in Euclidean space. These are similar to the so-called Christ-David Cubes (\cite{Dav88,Chr90}) in some respects. Both collections have the property that, much like dyadic cubes in Euclidean space, they can be divided into different generations and the cubes from each scale partition the cubes from previous generations. The main differences are that the Christ-David cubes and dyadic cubes are partitioned by cubes at the next generation of roughly the same size, while the children of Schul's cubes can vary wildly. Moreover, the Christ-David construction can be modified slightly to exactly partition a doubling space $X$, whereas Schul's cubes may not. The important property they do have, however, is that they are approximately like balls.

Fix $M,K>0$ and $c\in (0,\frac{1}{8})$. For each integer $n\geq 0$, let  $X_{n}\subseteq X$ be a  sequence of maximal $KM^{-n}$-nets in $X$. Let
\[\cB_{n}=\{B(x,KM^{-n}): x\in X_{n}\}, \;\; \cB=\bigcup_{n} \cB_{n}.\]
For $B=B(x,KM^{-n})\in \cB_{n}$, define
\[Q_{B}^{0}=cB, \;\; Q_{B}^{j}=Q_{B}^{j-1}\cup\bigcup\{cB: B\in \bigcup_{m\geq n} \cB_{m}, cB\cap Q_{B}^{j-1}\neq\emptyset\},\]
and
\[
Q_{B}=\bigcup_{j=0}^{\infty} Q_{B}^{j}.
\]
Basically, $Q_{B}$ is the union of all balls $B'$ that may be connected to $B$ by a chain $\{cB_{j}\}$ with $B_{j}\in \cB$, $\diam B_{j}\leq \diam B$, and $cB_{j}\cap cB_{j+1}\neq \emptyset$ for all $j$. 

For such a cube $Q$ constructed from $B(x,KM^{-n})$, we let $x_{Q}=x$ and $B_{Q}=B(x,cKM^{-n})$. 

Let 
\[\Delta_{n}=\{Q_{B}:B\in \cB_{n}\},  \;\; \Delta=\bigcup \Delta_{n}.\] 
Note that, for $Q\in \Delta_{n}$, $x_{Q}\in X_{n}$.

\begin{lemma}
If $c\in (0,\frac{1}{8})$, then for $X$ and $\Delta$ as above, the family of cubes $\Delta$ satisfy the following properties. 
\begin{enumerate}
\item If $Q,R\in \Delta$ and $Q\cap R\neq\emptyset$, then $Q\subseteq R$ or $R\subseteq Q$.
\item For $Q\in \Delta$, 
\begin{equation}
B_{Q}\subseteq Q\subseteq (1+8M^{-1})B_{Q}.
\label{e:1+2N^-1}
\end{equation}
\end{enumerate}
\label{l:cubes}
\end{lemma}

In other words, for $M$ large, our cubes don't differ much from balls. 

This version is a slight modification of a similar result in \cite[Theorem 3.19]{Sch07-TST} and is proven in \cite[Lemma 2.1]{Azz15}. There it is assumed that the $X_n$ are nested maximal $M^{-n}$-nets, but this is not necessary in the proof. In both papers it is also assumed that $K=1$, but the result above follows by just applying these results to a scaled copy of $X$.

\subsection{Now the proof}

The rest of this section is devoted to the proof of Theorem \ref{t:main}. Let $A\geq 1$ and let $X$ be a $C$-doubling metric space such that $\mu(X) = \cH^{s}(X)<\infty$. We will assume without loss of generality that $\diam X=1$.  Let $X_{n}$ be a sequence of $2^{-n}$-separated points in $X$, that is, where $|x-y|\geq 2^{-n}$ for all $x,y\in X$. Let 
\[
\cB_{n}=\{B(x,2^{-n}):x\in X_{n}\},\;\; \cB=\bigcup_{n\geq 0} \cB_{n}. 
\]

We would like to use Schul's cubes in such a way that each cube $Q$ corresponds to a dilated ball $AB$ for some $B\in \cB$. The issue here is that we constructed those cubes from contractions of balls and not enargements, i.e. using balls of the form $cB$ where $c\ll 1$, not balls $AB$ with $A>1$. What we do is split up the collections of balls into separate familes that are separated enough so that, if we consider balls $B'$ from one such family, then $cB'=AB$ for some $B$ in our original collection (this is the {\it thinning} process done in \cite[Section 3.3.1]{Sch07-TST}).

Let $a\in \bN$ be so that 
\begin{equation}
\label{e:A<2^a}
2^{a-1}\leq A<2^{a}.
\end{equation}
Since $X$ is doubling, one can find $N=N(A,C)$ and subsets $X_{n}^{1},...,X_{n}^{N}$ in $X_{n}$ that are  maximally $2^{-n+a+4}$-separated in $X_{n}$ and so that 
\[
X_{n}=\bigcup_{i=1}^{N} X_{n}^{i}.
\]
Let $J\in \bN$ be such that 
\begin{equation}
\label{e:Jrho}
2^{-J}<\min\left\{\rho,\frac{\ve}{16s}\right\}<2^{-J+1}
\end{equation}
Let $X_{n}^{i,j}=X_{nJ+j}^{i}$.  
For $i=1,,...,N$, $j=1,...J-1$, let $\Delta_{n}^{i,j}$ and $\Delta^{i,j}=\bigcup_{n} \Delta_{n}^{i,j}$ be those cubes constructed in the previous section for the sequence $(X_{n}^{i,j})_{n}$ of $2^{-nJ-j+a+4}$-separated points with $K=2^{-j+a+4}$, $M=2^{J}$, and
\[
c=A2^{-4-a}<2^{-4}<1/8,
\]
so that if
\[
\cB_{n}^{i,j}=\{B(x,2^{-nJ-j+a+4}):x\in X_{n}^{i,j}\}, \;\; \cB^{i,j}=\bigcup_{n\geq 0} \cB_{n}^{i,j},\]
and if 
\[B=B(x,2^{-nJ-j+a+4})=B(x,KM^{-n})\in \cB_{n},\] 
then
\[
cB=B(x,A2^{-nJ+j}). 
\]

So we have for $j=1,...,J-1$,
\[
A\cB_{nJ+j}
:= \{B(x,A2^{-n}):x\in X_{n}\}
=\bigcup_{i} c\cB_{n}^{i,j}=\bigcup_{i}\{cB:B\in \cB_{n}^{i,j}\}.
\]
and thus
\[
A\cB=\bigcup_{n,i,j} c \cB_{n}^{i,j}.
\]

Fix some $i$ and $j$. 

\begin{lemma}
For $\mu$-a.e. $x\in X$, if $x$ is contained in infinitely many $Q\in \Delta^{i,j}$, then
\begin{equation}
\label{e:uplim}
\lim_{r\rightarrow 0}\sup_{Q\in \Delta^{i,j}\atop x\in Q\subseteq B(x,r)} \frac{\mu(Q)}{(\diam Q)^{s}}\leq 1. 
\end{equation}
\end{lemma}

\begin{proof}
The proof is exactly the same as the analogous one with balls in place of cubes \cite[Theorem 6.2]{Mattila}, apart from the fact that we don't have the Besicovitch covering lemma, but this is not needed if we are working with cubes. We include the proof for completeness: 

Let $t>1$ and
\[
E_{t}=\left\{x\in X: \lim_{r\rightarrow 0}\sup_{Q\in \Delta^{i,j}\atop x\in Q\subseteq B(x,r)} \frac{\mu(Q)}{(\diam Q)^{s}}> t\right\},
\]
Assume $\mu(E_{t})>0$ for some $t>1$. Since $\mu(X)<\infty$, $\mu|_X$ is Radon (see Theorem \cite[Theorem 1.11 and 4.2]{Mattila}), so we may find $U\supseteq  E_{t}$ open with 
\begin{equation}
\label{e:1-tmu}
\mu(U\backslash E_{t})<(t-1)\mu(E_{t}).
\end{equation}
For any $\ve>0$ and for each $x\in E_{t}$, we may pick $Q(x)\subseteq U$ with $\diam Q(x)<\ve$ and $\frac{\mu(Q(x))}{(\diam Q(x))^{s}}> t$. Let $Q_{k}$ be the collection of maximal cubes we get in this way, so $E_{t}\subseteq \bigcup_{k} Q_{k}$. Hence,
\[
t\cH^{s}_{\ve}(E_{t})
\leq t\sum_{k} (\diam Q_{k})^{s}
<\sum_{k} \mu(Q_{k})
\leq \mu(U)
\]
thus, letting $\ve\rightarrow 0$, we get 
\[
t\mu(E_{t})
=\lim_{\ve \rightarrow 0}t\cH^{s}_{\ve}(E_{t})
\leq \mu(U)
\stackrel{\eqref{e:1-tmu}}{<} \mu(E_{t}).\]
which is impossible, thus $\mu(E_{t})=0$ for all $t>1$, which proves \eqref{e:uplim}. 
\end{proof}

The proof of Theorem \ref{t:main} now proceeds almost exactly as in \cite[Lemma 3.25]{Sch07-TST}. 

\begin{lemma}
Let
\[
\cC^{i,j}=\{Q\in \Delta^{i,j}: \cH^{s}_{\rho \diam Q}(X\cap Q)>(1+\ve/4)(\diam Q)^{s}\}
\]
and 
\[
\cC =  \left\{ B\in \cB: \cH^{s}_{\rho r_{B}}(X\cap AB)>(1+\ve)(2r_{B})^{s}\right\}.
\]
Then
\begin{equation}
\label{e:QBinC}
\{Q_{B}:B\in \cC, AB\in  c\cB^{i,j}\} \subseteq \cC^{i,j}. 
\end{equation}
\end{lemma}

\begin{proof}

Let $B\in \cC$ be such that $AB\in c\cB^{i,j}$. Recall that by \eqref{e:1+2N^-1} that $Q_{B}\supseteq B$ and so $\diam Q_{B}\geq \diam B=2B$. Thus, using that $(1+t)^{-s}\geq 1-st$ for $t\geq 0$, 
\begin{align*}
\cH_{\rho \diam Q_{B}}^{s}(X\cap Q_{B})
& \stackrel{\eqref{e:1+2N^-1}\atop \eqref{e:Hdecreasing}}{\geq} \cH_{\rho 2r_{B}}^{s}(X\cap B)
\geq (1+\ve)(\diam B)^{s}\\
& 
\stackrel{\eqref{e:1+2N^-1}}{\geq}  (1+\ve)(1+8M^{-1})^{-s}(\diam Q)^{s}\\
& = (1+\ve)(1+2^{-J+3})^{-s}(\diam Q)^{s}\\
& \geq (1+\ve)(1-s2^{-J+3})(\diam Q)^{s}\\
& \stackrel{\eqref{e:Jrho}}{ \geq} (1+\ve)\ps{1-\frac{\ve}{2}}
 \geq \ps{1+\frac{\ve}{4}}(\diam Q)^{s}
\end{align*}
\end{proof}

\begin{lemma}

For each $Q\in \cC^{i,j}$, we claim there is a function $w_{Q}$ defined on $X$, and a constant $\alpha >1 $ so that 
\begin{enumerate}
\item $\supp w_{Q}=0$ on $Q^{c}$,
\item $\int w_{Q}d\mu=(\diam Q)^{s}$
\item $w_{Q}(x)<\alpha^{-k_{Q}(x)}$ where $k_{Q}(x)$ is the number of cubes in $\cC^{i,j}$ properly contained in $Q$ containing $x$
\end{enumerate}
\end{lemma}

\begin{proof}
\def\Stop{{\rm Stop}}
For convenience, we will treat functions as measures below, so given a function $f$, $f(A)$ will also denote $\int_{A} fd\mu$. 

We will define $w_{Q}$ in a martingale fashion, that is, as a sequence of functions where we obtain the next function by redefining the previous function in various cubes so that the integrals in those cubes is unaffected. First we need to introduce some notation relating to the cubes. For  $Q\in \cC^{i,j}$, let $\Stop_{0}=\{Q\}$, $\Stop_{1}(Q)$ be the set of maximal cubes in $\cC^{i,j}$ properly contained in $Q$ and inductively set 
\[
\Stop_{k+1}(Q) = \bigcup_{R\in \Stop_{k}(Q)}\Stop_{1}(R).
\]
Now we define the sequence of functions that will converge to $w_Q$. We first let 
\[
w_{Q}^{0}=\one_{X\cap Q}\frac{(\diam Q)^{s}}{\mu(Q)},
\] 
so in this way, $w_{Q}^{0}(Q)=(\diam Q)^{s}$. 

Let
\[
R_{Q}=Q\backslash \bigcup_{R\in \Stop_1(Q)}R
\]
and
\[
m(Q)=\mu(R_{Q})+\sum_{R\in \Stop_1(Q)}(\diam R)^{s} . 
\]
Note that by \eqref{e:Jrho}, since cubes properly contained in $Q$ (and hence those cubes in $\Stop_{1}(Q)$) have diameter at most $2^{-J}\diam Q<\rho\diam Q$, setting
\begin{equation}
\label{e:alpha}
\alpha = 1+\frac{\ve}{4}
\end{equation}
 we have 
\begin{align}
\label{e:mQ>1+e}
m(Q) & \geq\cH^{s}_{\rho\diam Q}(R_{Q})+\sum_{R\in \Stop_1(Q)}(\diam R)^{s} \\
& \geq \cH^{s}_{\rho\diam Q}(X\cap Q)\geq \ps{1+\frac{\ve}{4}}(\diam Q)^{s}=\alpha(\diam Q)^{s}. \notag 
\end{align}
Now let $w_{Q}^{1}$ be a function on $X$ that is constant in the sets $R_{Q}$ and $R\in \Stop_{1}(Q)$ (and zero elsewhere) so that 
\[
w_{Q}^{1}(R_{Q}) = \frac{\mu(R_{Q})}{m(Q)} w_{Q}^{0}(Q) \;\; \mbox{ and }w_{Q}^{1}(R)=\frac{(\diam R)^{s}}{m(Q)}w_{Q}^{0}(Q) .
\]
In this way,
\[
w_{Q}^{1}(Q) = w_{Q}^{0}(Q) =(\diam Q)^{s}.
\]

Inductively, suppose for some $k\geq 1$ we have defined $w_{Q}^{k}$ for each $Q\in \cC^{i,j}$. We now let
\[
w_{Q}^{k+1}|_{R_{Q}} = w_{Q}^{k}|_{R_{Q}} \;\; \mbox{ and }
w_{Q}^{k+1}|_{R} = \frac{w_{Q}^{0}(Q) }{m(Q)}w_{R}^{k}|_{R} \;\; \mbox{ for $R\in \Stop_{1}(Q)$}.
\]
\begin{remark}
\label{r:wremark}
By construction, we have  for all $k$
\[
w_{Q}^{k}(R_{Q})=w_{Q}^{k-1}(R_{Q})=\cdots = w_{Q}^{1}(R_{Q}) = \frac{\mu(R_{Q})}{m(Q)} w_{Q}^{0}(Q),
\]
\[
w_{Q}^{k}(Q)=w_{Q}^{k-1}(Q)=\cdots  = w_{Q}^{0}(Q)=(\diam Q)^{s},
\]
and $w_{Q}^{k}$ is constant on each set $R_{T}$ for $T\in \bigcup_{\ell=0}^{k-1} \Stop_{\ell}(Q)$ and on $T\cap X$ for each $T\in \Stop_{k}(Q)$ (and is zero outside these sets). 
\end{remark}

We now claim that if $x\in Q$ is contained in $k_0$ many cubes from $\cC^{i,j}$ properly contained in $Q$
\begin{equation}
\label{e:w<1+e^k}
w_{Q}^{k_0}(x) < \alpha^{-k_0} .
\end{equation}

We begin the proof of the claim: First, since $w_{Q}^{0}$ is constant in $Q$, for $x\in Q$,
\[
w_{Q}^{0}(x) = \frac{(\diam Q)^{s}}{\mu(Q)}
\stackrel{\eqref{e:Hdecreasing}}{\leq}  \frac{(\diam Q)^{s}}{\cH^{s}_{\rho\diam Q}(X\cap Q)}
\stackrel{\eqref{e:mQ>1+e}}{<}\alpha^{-1}.
\]
This proves the $k=0$ case of \eqref{e:w<1+e^k}. For $k\geq 1$, let $T\in \Stop_k(Q)$. Then $T\in \Stop_{k-1}(R)$ for some $R\in \Stop_{1}(Q)$, and the construction implies
\begin{align}
 \frac{w_{Q}^{k}(T)}{(\diam T)^{s}}
& =  \frac{w_{Q}^{0}(Q)}{m(Q)}\frac{w_{R}^{k-1}(T)}{(\diam T)^{s}}
= \frac{(\diam Q)^{s}}{m(Q)}\frac{w_{R}^{k-1}(T)}{(\diam T)^{s}}
\stackrel{\eqref{e:mQ>1+e}}{<}\alpha^{-1}\frac{w_{R}^{k-1}(T)}{(\diam T)^{s}} \notag \\
& <\cdots < \alpha^{-k}\frac{w_{T}^{0}(T)}{(\diam T)^{s}}
=\alpha^{-k}
\label{e:w/T<...}
 \end{align}
 In particular, since $w_{Q}^{k}$ is constant on $T\cap X$, this shows that for $x\in T$,
\begin{align*}
w_{Q}^{k}(x)
&  =\frac{w_{Q}^{k}(T)}{\mu(T)}
\stackrel{\eqref{e:Hdecreasing}}{\leq} \frac{w_{Q}^{k}(T)}{\cH^{s}_{\rho\diam(T)}(X\cap T)}
\stackrel{\eqref{e:mQ>1+e}}{<} \alpha^{-1} \frac{w_{Q}^{k}(T)}{(\diam T)^{s}}\stackrel{\eqref{e:w/T<...}}{<}\alpha^{-k-1}.
\end{align*}
Moreover, if $U\in \Stop_{\ell}(Q)$ for some $1\leq \ell<k$, then since $w_{Q}^{k}$ is constant on $R_{U}$, for $x\in R_{U}$,
\begin{align*}
w_{Q}^{k}(x)
& =w_{Q}^{\ell}(x)
=\frac{w_{Q}^{\ell}(R_{U})}{\mu(U)}
\stackrel{\eqref{e:Hdecreasing}}{\leq}  \frac{w_{Q}^{\ell}(R_U)}{\cH^{s}_{\rho\diam(U)}(X\cap U)}
\stackrel{\eqref{e:mQ>1+e}}{<} \alpha^{-1} \frac{w_{Q}^{\ell}(R_U)}{(\diam U)^{s}} \\
& \leq \alpha^{-1} \frac{w_{Q}^{\ell}(U)}{(\diam U)^{s}} 
\stackrel{\eqref{e:w/T<...}}{<}\alpha^{-\ell-1}
\end{align*}
By Remark \ref{r:wremark}, any $x\in Q$ where $w_{Q}^{k}$ is nonzero is in either some $T\in \Stop_{k}(Q)$ or $R_{U}$ for some $U\in \Stop_{\ell}(Q)$, $\ell<k$, so the above estimates imply  \eqref{e:w<1+e^k} and prove the claim. 

In particular, $w_{Q}^{k}$ is a sequence of uniformly bounded $L^{\infty}$ functions vanishing outside of $X\cap Q$. By \eqref{e:uplim}, $\mu$-a.e. $x\in X$ is contained in at most finitely many $Q\in \cC^{i,j}$, and so $w_{Q}^{k}$ converges pointwise a.e. to a function $w_{Q}$ that is zero outside $Q^{c}$ (proving (1)), and by the dominated convergence theorem, 
\[
w_{Q}(R)=\lim_{k\rightarrow\infty}w_{Q}^{k}(R) \;\; \mbox{ for }R\in \bigcup_{k=0}^{\infty} \Stop_{k}(Q),
\]
proving (2). Finally, (3) follows from the previous claim.

\end{proof}

In particular, if $\cC_{0}^{i,j}$ are the maximal cubes in $\cC^{i,j}$ (since recall the sizes of the balls in $\cC^{i,j}$ are bounded above), then those cubes are disjoint and thus
\begin{align*}
\sum_{Q\in \cC^{i,j} }(\diam Q)^{s}
& =\sum_{Q\in \cC^{i,j} }\int w_{Q}(x)d\mu(x)
=\int \sum_{Q\in \cC^{i,j} } w_{Q}(x)d\mu(x)
\\
& <\sum_{Q_0\in \cC^{i,j}_{0}}\int_{Q_{0}}\sum_{k=0}^{\infty} \alpha^{-k} d\mu(x) 
=\frac{1}{1-\alpha}\sum_{Q_0\in \cC^{i,j}_{0}}\mu(Q_{0})
\lec \frac{\mu(X)}{\ve}.
\end{align*}

Hence, by our choice of $J$, and recalling the definition of $N$ from \eqref{e:A<2^a}.
\begin{align*}
\sum_{B\in \cC}(\diam B)^{s}
& \leq \sum_{i,j}\sum_{Q_{AB}\in c\cC^{i,j}}(\diam Q_{AB})^{s}
\leq \sum_{i,j}\sum_{Q\in c\cC^{i,j}}(\diam Q)^{s}\\
& \lec \sum_{i=1,...,N\atop j=1,...,J}\frac{\mu(X)}{\ve} 
\leq \frac{NJ}{\ve} \mu(X)
\lec N\frac{\log \frac{1}{\min\{\rho,\ve/s\}}}{\ve} \mu(X).
\end{align*}

\appendix

\section{Weak convergence of measures}

We now prove the results stated in Section \ref{s:WCM}. We begin by recalling the definitions of Choquet integration and weak convergence of measures. Recall that for us, a measure is a monotonic, countably subadditive set function which vanishes for the empty set. 

For a measure $\mu$ and a function $f: \R^n \rightarrow [0,\infty)$, define the Choquet integral of $f$ with respect to $\mu$ by the formula
\[
\int f \, d\mu = \int_0^\infty \mu( \{x \in \R^n : f(x) > t \}) \, dt.
\]
For a real valued function $f: \R^n \rightarrow \R,$ let $f^+ = \max\{f,0\}$ and $f^{-} = \max \{-f,0 \}.$ Define the Choquet integral of $f$ with respect to $\mu$ by 
\[
\int f \, d\mu = \int f^+ \, d\mu - \int f^{-} \, d\mu. 
\]

For a measure $\mu,$ the Choquet integral with respect to $\mu$ is not necessarily additive or even subadditive. We do however have the following quasi-subadditivity. 

\begin{lemma}\label{l:quasiadd}
	Let $0<\gamma <1, \ \mu$ a measure and $f,g : \R^n \rightarrow [0,\infty).$ Then 
	\begin{align}\label{e:subadd}
		\int (f + g) \, d\mu \leq \frac{1}{\gamma} \int f \, d\mu + \frac{1}{1-\gamma} \int g \, d \mu.
	\end{align}
\end{lemma}

\begin{proof}
	For any $t \geq 0$ we have
	\[
	\{x \in \R^n : f(x) + g(x) > t \} \subseteq \{x \in \R^n : f(x) > \gamma t \} \cup \{x \in \R^n : g(x) > (1-\gamma)t \},
	\]
	since outside this union, $f(x) + g(x) \leq \gamma t + (1-\gamma)t = t.$ The lemma follows immediately by using the sub-additivity of $\mu$ and integrating in $t$. 
\end{proof}

\begin{definition}
	Let $\{\mu_k\}$ be a sequence of measures on $\R^n.$ We say the sequence $\{\mu_k\}$ converges \textit{weakly} to a Radon measure $\mu$, and write 
	\[
	\mu_k \rightharpoonup \mu,
	\]
	if
	\[
	\lim_{k \rightarrow \infty} \int \vp \, d\mu_k = \int \vp \, d\mu \quad \text{for all} \ \vp \in C_0(\R^n). 
	\]
	Here, $C_0(\R^n)$ is the space of continuous functions of compact support. 
\end{definition}

We can now prove the main results from Section \ref{s:WCM}, we shall state each result again before proving it. 

\begin{lemma}\label{l:semicont}
	Suppose $\{\mu_k \}$ is a sequence of measures converging weakly to a Radon measure $\mu.$ For $K \subseteq \R^n$ compact and $U \subseteq \R^n$ open we have 
	\[
	\mu(K) \geq \limsup_{k \rightarrow \infty} \mu_k(K)
	\]
	and
	\[
	\mu(U) \leq \liminf_{k \rightarrow \infty} \mu_k(U). 
	\]
\end{lemma}

\begin{proof}
	
	Let $\ve >0.$ Since $\mu$ is Radon, there exists and open set $V \supset K$ such that $\mu(V) \leq \mu(K) + \ve.$ By Urysohn's Lemma, there is $\vp \in C_0(\R^n)$ such that $0 \leq \vp \leq 1, \ \vp \equiv 1$ on $K$ and $\text{supp} \, \vp \subset V.$ Then
	\begin{align*}
		\mu(K) &\geq \mu(V) - \ve \geq \int \vp \, d\mu - \ve = \lim_{k \rightarrow \infty} \int \vp \, d\mu_k - \ve \\
		&\geq \limsup_{k \rightarrow \infty} \mu_k(K) - \ve. 
	\end{align*}
	Similarly, there exist a compact set $F \subset U$ such that $\mu(F) \geq \mu(U) - \ve.$ We can find $\vp \in C_0(\R^n)$ such that $0 \leq \vp \leq 1, \ \vp \equiv 1$ on $F$ and $\text{supp} \, \vp \subset U.$ Then 
	\begin{align*}
		\mu(U) &\leq \mu(F) + \ve \leq \int \vp \, d\mu + \ve = \lim_{k \rightarrow \infty} \int \vp \, d\mu_k +\ve \\
		& \leq \liminf_{k \rightarrow \infty} \mu_k(U) + \ve.
	\end{align*}
	The result follows since $\ve$ was arbitrary. 
	
\end{proof}

\begin{lemma}\label{l:converging}
	Suppose $\{\mu_k\}$ is a sequence of measures converging weakly to a Radon measure $\mu$. Suppose additionally there exists $C_0 >0$ such that each $\mu_k$ is $C_0$-Ahlfors $d$-regular (in the sense that it satisfies the upper and lower regularity condition with constant $C_0$, but may not be additive). Then, for any ball $B$, we have 
	\[
	\lim_{k \rightarrow \infty} \left( \sup_{p \in B \cap \emph{supp} \, \mu} \emph{dist}(p, \emph{supp} \, \mu_k) \right) = 0
	\]
	and
	\[
	\lim_{k \rightarrow \infty} \left(\sup_{p \in B \cap \emph{supp} \, \mu_k} \emph{dist}(p, \emph{supp} \, \mu) \right) = 0.
	\]
\end{lemma}

\begin{proof}[Proof of Lemma \ref{l:converging}]
Let $\delta >0.$ Let $K \in \bN$ and suppose there exists $p \in B \cap \text{supp} \, \mu$ such that $\text{dist}(p, \text{supp} \, \mu_k) > \delta$ for all $k \geq K$. Let $\phi \in C_0(\R^n)$ be such that $0 \leq \phi \leq 1$, $\phi \equiv 1$ on $B(p,\delta/2)$ and $\text{supp} \, \phi \subset B(p,\delta).$ Since $p \in \text{supp} \, \mu,$
\[
\int \phi \, d\mu \geq \mu(B(p,\delta/2)) > 0,
\]
but
\[
\int \phi \, d\mu_k = 0
\]
for all $k \geq K,$ which gives a contradiction. 

The proof of the second equality is lifted verbatim from the proof of \cite[Lemma III.2.43]{of-and-on}. Let $\ve >0$ and $B_1, \dots, B_\ell$ be a finite collection of balls of radius $\ve$ which cover $B$. For $i = 1,\dots, \ell$, let $\phi_i \in C_0(\R^n)$ satisfy $\phi_i \equiv 1$ on $2B_i$ and $\text{supp} \, \phi_i \subseteq 3B_i.$ Choose $K$ large enough so that
\[
\left|\int \phi_i \, d\mu_k - \int \phi_i \, d\mu\right| \leq (2C_0)^{-1}\ve^d
\]
for all $k \geq K$ and $i = 1, \dots \ell.$ For each such $i$ and $k,$ if $B_i$ intersects $\text{supp} \, \mu_k$ then
\[
\int \phi_i \, d\mu_k \geq C_0^{-1} \ve^d
\]
hence
\[
\int \phi_i \, d\mu \geq (2C_0)^{-1}\ve^d,
\]
which in turn implies $3B_i$ intersects $\text{supp} \, \mu.$ Thus, if $p \in \text{supp} \, \mu_k$ for $k \geq K,$ then $\text{dist}(p,\text{supp} \, \mu) \leq 6\ve.$ Since $\ve$ was arbitrary this implies the second equality. 
\end{proof}

Finally, we must prove the following.

\begin{lemma}\label{l:subseq}
	Let $\{E_k\}$ be a sequence of $C_0$-Ahlfors $d$-regular sets in $\R^n$ and $\{\rho_k\}$ a sequence of positive real numbers such that $\rho_k \rightarrow 0.$ Let $\mu_k = \cH^d_{\rho_k}|_{E_k},$ then there exists sub-sequence $\{\mu_{k_j}\}$ and a Radon measure $\mu$ such that $\mu_{k_j} \rightharpoonup \mu.$
\end{lemma}

Before proving Lemma \ref{l:subseq}, we need a series of lemmas. Let $\mathscr{I}$ the collection of Euclidean dyadic cubes in $\R^n$ and $\cI_m$ be those cubes in $\cI$ with side length $2^{-m},$ for $m \in \bZ.$ Let $G^m$ denote the dyadic grid at scale $m$, that is, 
\[
G^m = \bigcup_{I \in \cI_m} \partial I.
\]
For $x \in \R^n$, let $G_x^m = x + G^m$ denote the translate of the dyadic grid at scale $m$ by $x.$
\begin{lemma}\label{lem:grid}
	Let $\delta >0,$ $m \in \bN$, $R>0$ and $\mu$ a Radon measure. Then, there exists $x \in \R^n$ such that 
	\[
	\mu( G_x^m \cap B(0,R)) < \delta.
	\]
\end{lemma}

\begin{proof}
	Assume the lemma is false. Let $x_0 = (1,1,\dots,1) \in \R^n.$ By assumption, we can find a sequence of distinct real numbers $0\leq \lambda_k \leq 2^{-m}$, such that
	\[
	\mu (G_{\lambda_k x_0}^m \cap B(0,R)) \geq \delta
	\]
	for each $k.$ Let $x_k = \lambda_kx_0.$ Notice that $G^m_{x_i} \cap  G^m_{x_j} \cap  G^m_{x_k} = \emptyset$ for $i\not=j\not=k,$ that is, the $G_{x_i}^m$ have bounded overlap. Hence, 
	\[
	\infty = \sum_{i=1}^\infty \mu( G^m_{x_i} \cap B(0,R)) \leq 2\mu(B(0,R)) \lesssim 1,
	\]
	which is a contradiction. 
\end{proof}

\begin{lemma}\label{l:nhood}
	Let $\delta>0, \ m \in \bN,$ $R >0$, $\mu$ be a Radon measure,  and $x \in \R^n.$ If $\mu(G_x^m \cap B(0,R)) < \delta,$ then there exists $\eta >0$ such that 
	\[
	\mu( G_x^m(\eta) \cap B(0,R)) < 2\delta,
	\]
	where $G_{x}^m(\eta)$ denotes the closed $\eta$-neighbourhood of $G_x^m.$
\end{lemma}

\begin{proof}
	This simply follows by taking a sequence $\eta_j \downarrow 0$ and using the continuity property of $\mu$ on decreasing sequences of sets. 
\end{proof}
Let $\mu_k$ be as in Lemma \ref{l:subseq} and set $\tilde{\mu}_k = \cH^d|_{E_k}.$ Note, $\mu_k \leq \tilde{\mu}_k$ for each $k \in \bN.$ Since each $\tilde{\mu}_k$ is a Radon measure and $\sup_k \tilde\mu_k(K) < \infty$ for all compact $K \subseteq \R^n$ (by virtue of the $E_k$ being $C_0$-Ahlfors $d$-regular), we are able to extract a weakly convergent subsequence. Therefore, without loss of generality, we may assume $\tilde{\mu}_k \rightharpoonup \tilde{\mu}$ to some Ahlfors $d$-regular Radon measure $\tilde{\mu}.$

For $i \in \bN,$ let $\phi_i$ be a $C^\infty$-bump function so that $0 \leq \phi_i \leq 1, \ \phi_i \equiv 1$ on $B(0,i)$ and $\text{supp} \, \phi_i \subseteq B(0,i+1).$ Let
\begin{align*}
	D' = \{P\phi_i : P \ \text{is a non-negative polynomial}& \\ &\hspace{-4em}\text{with rational coefficients and} \ i \in \bN \}
\end{align*}
and let $D$ be the set of all rational finite linear combinations of $D'.$ By the Weierstrass Approximation Theorem, it follows that $D'$ forms a countable dense subset of $C_0^+(\R^n)$ under $||\cdot||_\infty.$  Clearly then, this is also true for $D$.

Let $\phi \in D.$ Since $E_k$ is $C_0$-Ahlfors $d$-regular for each $k \in \bN,$ it follows that 
$$\mu_k(\text{supp} \, \phi) \leq \tilde{\mu}_k(\text{supp} \, \phi) \leq C_0 (\text{diam}( \text{supp} \, \phi)/2)^d <\infty.$$
Then, since $\phi$ is bounded, we can extract a convergent subsequence of $\{\int \phi \,  d\mu_k\}.$ We claim we can extract a further subsequence so that 
\[
L\phi = \lim_{k \rightarrow \infty} \int \phi \, d\mu_k \ \text{exists for all} \ \phi \in D. 
\]
This follows by a diagonal argument: enumerate $D=\{\phi_{1},...\}$. Pick  a subsequence $n_{k}^{1}$ so that $\int \phi_{1}d\mu_{n_{k}^{1}}$ converges. Now pick a subsequence $n_{k}^{2}$ of $n_{k}^{1}$ so that $\int \phi_{2}d\mu_{n_{k}^{2}}$ converges, and inductively, given a subsequence $n_{k}^{j}$, pick  a subsequence $n_{k}^{j+1}$ of this sequence so that $\int \phi_{j+1} d\mu_{n_{k}^{j+1}}$ converges. Now set $n_{k}=n_{k}^{k}$. Then for each $j$, $n_{k}^{k}$ is a subsequence of $n_{k}^{j}$ for $k>j$, and the limit above converges for every $\phi_{j}$, which proves the claim.

We will show that $L$ defines a linear functional on all of $C_0(\R^n)$. We first treat the case of non-negative functions.

\begin{lemma}\label{l:step}
	Let $k \geq 0$ and $N \geq 1.$ For a function $\phi$ of the form 
	\[
	\phi = \sum_{j = 1}^N a_j \mathds{1}_{A_j}
	\]
	where $a_j \geq 0$ and $A_j \subseteq \R^n$ are such that $\emph{dist}(A_i,A_j) \geq 2\rho_k$, we have 
	\[
	\int \phi  \, d\mu_k = \sum_{j=1}^N a_j \mu_k(A_j).  
	\]
\end{lemma}

\begin{proof}
	We claim $\mu_k$ is additive on any subset of $\{A_j\},$ i.e. for any $\mathscr{C} \subseteq \{A_j\},$ we have
	\[
	\mu_k \left( \bigcup_{A_j \in \mathscr{C}} A_{j} \right) = \sum_{A_j \in \mathscr{C}} \mu_k(A_{j}). 
	\]
	Let $\mathscr{C} \subseteq \{A_j \}.$ The forward inequality is immediate by sub-additivity. To prove the reverse inequality, let $\ve >0$ and suppose $\mathscr{U}$ is a countable cover for $\bigcup_{A_j \in \mathscr{C}}A_j$ such that $\text{diam}(U) \leq \rho_k$ for each $U \in \mathscr{U}$ and
	\[
	\sum_{U \in \mathscr{U}} \text{diam}(U)^d \leq \mu_k \left( \bigcup_{A_j \in \mathscr{C}} A_j \right) + \ve. 
	\]
	Since the $A_j$ are separated by $2\rho_k$ and $\text{diam}(U) \leq \rho_k$, each $U$ intersects only a single $A_j.$ Hence, the sets $\mathscr{U}_j = \{U : U \cap A_j \not=\emptyset\}$ form a partition of $\mathscr{U}.$ Then,
	\[
	\sum_{A_j \in \mathscr{C}} \mu_k(A_j) \leq \sum_{A_j \in \mathscr{C}} \sum_{U\in \mathscr{U}_j} \text{diam}(U)^d = \sum_{U \in \mathscr{U}} \text{diam}(U)^d  \leq \mu_k \left( \bigcup_{A_j \in \mathscr{C}} A_j \right) + \ve,
	\]
	which proves the claim. 
	
	Let $\phi$ be as above. We may assume $a_{j+1} \geq a_j$ for all $j=0,1,\dots$, where we define $a_0=0.$ Then
	\begin{align*}
		\int \phi \, d\mu_k &= \int_0^\infty \mu_k (\{x \in \R^n : \phi(x) > t \} ) \, dt \\
		&= \sum_{j=1}^N \int_{a_{j-1}}^{a_j} \mu_k ( \{ x \in \R^n : \phi(x) > t \}) \, dt \\
		&= \sum_{j=1}^N (a_j - a_{j-1}) \mu_k \left( \bigcup_{i = j}^N A_i \right) = \sum_{j=1}^N (a_j -a_{j-1}) \sum_{i=j}^N\mu_k(A_i) \\
		&=  \sum_{i=1}^N\mu_k(A_i)  \sum_{j=1}^i (a_j -a_{j-1})  =\sum_{i=1}^N a_i \mu_k(A_i). 
	\end{align*} 
\end{proof}
As an immediate consequence of Lemma \ref{l:step}, we get the following.
\begin{corollary}\label{c:additivity}
	Let $k \geq0$ and $N \geq 1.$ Suppose $\phi$ and $\vp$ are functions of the form 
	\[
	\phi = \sum_{j = 1}^N a_j \mathds{1}_{A_j} \quad \text{and} \quad \vp = \sum_{j = 1}^N a^\prime_j \mathds{1}_{A_j}
	\]
	where $a_j,a_j^\prime \geq 0$ and $A_j \subseteq \R^n$ are such that $\emph{dist}(A_i,A_j) \geq 2\rho_k.$ Then 
	\[
	\int (\phi + \vp) \, d\mu_k = \int \phi \, d\mu_k + \int \vp \, d\mu_k. 
	\]
\end{corollary}

\begin{lemma}\label{l:linear}
	Let $\phi,\vp \in D,$ such that $\phi,\vp \geq 0.$ Then 
	\[
	L(\phi + \vp) = L\phi + L\vp. 
	\]
\end{lemma}

\begin{proof}
	Let $\alpha >0$ (to be chosen small later) and choose $m =m(\alpha)$ large enough so that if $x,y \in I \in \cI_m$ then $|\phi(x) - \phi(y)| \leq \alpha$ and $|\vp(x) - \vp(y)| \leq \alpha.$ This is possible since $\phi$ and $\vp$ are $C^\infty$ function with compact support and so they have bounded derivatives. Set 
	\[
	M = \sup_{x \in \R^n} \max\{\phi(x),\vp(x)\}
	\]
	and let $R> 0$ be such that $\supp \phi,\supp \vp \subseteq B(0,R).$ For each $k \geq 0,$ since $E_k$ is $C_0$-Ahlfors regular, we have
	\begin{equation}
		\label{e:mu_k<C_0R^d}
		\mu_k(B(0,R)) \leq C_0R^d.
	\end{equation}
	Let $\delta >0$ (to be chosen small enough later). Recall the definition of $\tilde{\mu},$ just after the statement of Lemma \ref{l:nhood}. By Lemma \ref{lem:grid}, we can find a translate of the dyadic grid $G^m = G_x^m$ such that 
	\[
	\tilde{\mu}( G^m \cap B(0,R)) < \delta.
	\]
	By Lemma \ref{l:nhood}, we can choose $\eta >0$ small enough so that 
	\[
	\tilde{\mu}( G^m(2^{-m} \eta ) \cap B(0,R)) < 2\delta.
	\]
	Since $\tilde{\mu}$ is the weak limit of the $\tilde{\mu}_k,$ there exist $K = K(\delta)$ such that for $k \geq K$, 
	\begin{align*}
		\tilde{\mu}_k( G^m(2^{-m} \eta ) \cap B(0,R)) &< \limsup_{n \rightarrow \infty}\tilde{\mu}_n( G^m(2^{-m} \eta ) \cap B(0,R)) +\delta \\
		&\leq \tilde{\mu}(G^m(2^{-m} \eta) \cap B(0,R))+\delta \\
		&\leq 3\delta.
	\end{align*}
	Since $\mu_k \leq \tilde{\mu}_k$ for all $k,$ this remains true for the $\mu_k,$ that is, for $k \geq K,$
	\begin{equation}\label{e:mu_kG^m<3delta}
		\mu_k(G^m(2^{-m}\eta) \cap B(0,R)) \leq 3\delta.
	\end{equation}
	For $I \in \cI,$ let $\phi_I = \phi \mathds{1}_{(1-\eta)I}.$ We can write
	\[
	\phi = \sum_{I \in \cI_m} \phi_I + \left(\phi - \sum_{I \in \cI_m} \phi_I\right) = \sum_{I \in \cI_m} \phi_I + \phi_{ G}. 
	\]
	Notice that $\phi_{ G}$ is supported on $ G^m(2^{-m} \eta) \cap B(0,R).$ Define also $\tilde{\phi}_I : \R^n \rightarrow \R,$ such that
	\[
	\tilde{\phi}_I(x) = \inf_{y \in I} \phi(y) \one_{(1-\eta)I}(x) .
	\]
	By our choice of $m$, $|\phi_I - \tilde{\phi}_I|  \leq \alpha$ for each $I \in \cI_m.$ Hence
	\[
	\phi \leq \sum_{I \in \cI_m} (\tilde{\phi}_I + \alpha \one_{(1-\eta)I \cap B(0,R)}) + \phi_{G} \leq \alpha \mathds{1}_{B(0,R)} +  \sum_{I \in \cI_m} \tilde{\phi}_I + \phi_{G}
	\]
	Similarly, we define $\vp_I, \ \tilde{\vp}_I$ and $\vp_{ G},$ to get 
	\[
	\vp \leq \alpha \mathds{1}_{B(0,R)} + \sum_{I \in \cI_m} \tilde{\vp}_I + \vp_{G}.  
	\]
	For any $I,I^\prime \in \cI_m,$ we have 
	\[\text{dist}((1-\eta)I,(1-\eta)I^\prime) \geq 2^{-m}\eta.\]
	Thus, for $k$ large enough so that $2\rho_k \leq 2^{-m}\eta$, by Corollary \ref{c:additivity},  
	\begin{align}\label{e:add}
		\int \sum_{I \in \cI_m} \tilde{\phi}_I \, d\mu_k + \int \sum_{I \in \cI_m} \tilde{\vp}_I \, d\mu_k = \int \left(\sum_{I \in \cI_m} \tilde{\phi}_I + \sum_{I \in \cI_m} \tilde{\vp}_I \right)\, d\mu_k.
	\end{align}
	Let $\ve >0,$ $\gamma \in (0,1)$ and suppose $\alpha$ and $\delta$ have been chosen small enough so that $\gamma \ve \geq 4C_0 R^d\alpha + 12M\delta.$ Using the above combined with Lemma \ref{l:quasiadd} we can find $K = K(\eta,\delta,m)$ such that for $k \geq K,$
	\begin{align*}
		\int \phi \, d\mu_k + \int \vp \, d\mu_k 
		& \stackrel{\eqref{e:subadd}}{\leq} \frac{1}{1-\gamma} \left[ \int \sum_{I \in \cI_m} \tilde{\phi}_I \, d\mu_k + \int \sum_{I \in \cI_m} \tilde{\vp}_I \, d\mu_k\right]  \\
		& + \frac{1}{\gamma} \left[\int (\alpha\one_{B(0,R)} + \phi_{ G}) \, d\mu_k +\int (\alpha\one_{B(0,R)} + \vp_{G} )\, d\mu_k \right] \\
		& \stackrel{\eqref{e:add} \atop \eqref{e:subadd}, (\gamma =1/2)}{\leq}\frac{1}{1-\gamma} \left[ \int \left(\sum_{I \in \cI_m} \tilde{\phi}_I + \sum_{I \in \cI_m} \tilde{\vp}_I \right)\, d\mu_k\right] \\
		& + \frac{2}{\gamma}\left[ 2\int \alpha\mathds{1}_{B(0,R)} \, d\mu_k + \int \phi_{ G} d\mu_k + \int \vp_{ G} \, d\mu_k \right] \\
		&\leq \frac{1}{1-\gamma} \int (\phi + \vp )\, d\mu_k \\
		& + \frac{2}{\gamma}\left[ 2\alpha\mu_k(B(0,R)) + 2M\mu_k(G^m(2^{-m} \eta) \cap B(0,R)) \right] \\
		&\stackrel{\eqref{e:mu_k<C_0R^d} \atop \eqref{e:mu_kG^m<3delta}}{\leq} \frac{1}{1-\gamma} \int (\phi + \vp) \, d\mu_k + \frac{2}{\gamma} \left[ 2C_0R^d\alpha + 6M\delta \right] \\
		&\leq \frac{1}{1-\gamma} \int (\phi + \vp) \, d\mu_k  + \ve. 
	\end{align*}
	
	\begingroup
	\allowdisplaybreaks
	
	On the other hand
	\begin{align*}
		\int (\phi + \vp) \, d\mu_k &\leq \int ( 2\alpha\mathds{1}_{B(0,R)} + \sum_{I \in \cI_m} \tilde{\phi}_I + \phi_{G} + \sum_{I \in \cI_m} \tilde{\vp}_I + \vp_{G} )  \, d\mu_k \\
		& \stackrel{\eqref{e:subadd}}{\leq} \frac{1}{1-\gamma}\int \ps{ \sum_{I \in \cD_m} \tilde{\phi}_I + \sum_{I \in \cI_m} \tilde{\vp}_I }\, d\mu_k  \\
		& \hspace{1em} + \frac{1}{\gamma} \int( 2\alpha \mathds{1}_{B(0,R)} + \phi_G + \vp_G) \, d\mu_k \\
		& \stackrel{\eqref{e:subadd}, (\gamma = 1/2)}{\leq}  \frac{1}{1-\gamma} \left[ \int \sum_{I \in \cI_m} \tilde{\phi}_I \, d\mu_k + \int \sum_{I \in \cI_m} \tilde{\vp}_I \, d\mu_k\right]  \\
		& \hspace{1em} +\frac{2}{\gamma} \left[  \int 2\alpha \mathds{1}_{B(0,R)} \, d\mu_k + \int (\phi_G + \vp_G )\, d\mu_k \right] \\
		& \leq \frac{1}{1-\gamma} \left[ \int \phi \, d\mu_k + \int \vp \, d\mu_k\right]   \\
		& \hspace{1em}  +\frac{2}{\gamma}\left[ 2\alpha\mu_k(B(0,R)) + 2M\mu_k(G^m(2^{-m}\eta ) \cap B(0,R)) \right] \\
		& \leq  \frac{1}{1-\gamma} \left[ \int \phi \, d\mu_k + \int \vp \, d\mu_k\right]   + \ve. 
	\end{align*}
	Taking $k \rightarrow \infty$ in the previous two sequences of inequalities, we get 
	\[ L(\phi) + L(\vp) \leq \frac{1}{1-\gamma} L(\phi + \vp) + \ve \]
	and 
	\[ L(\phi + \vp) \leq \frac{1}{1-\gamma} \left(L(\phi) + L(\vp) \right) + \ve \]
	Thus, taking $\ve, \gamma \rightarrow 0$ we have
	\[ L(\phi + \vp) = L(\phi) + L(\vp) \]
	and this finishes the proof.
\end{proof}

\endgroup

\begin{lemma}\label{l:continuity}
	Let $f \in C_0^+(\R^n)$ and $R > 0$ be such that $\supp f \subseteq B(0,R).$  Suppose $\{f_i\}$ is a decreasing sequence of functions in $D$ such that $f_i(x) \geq f(x)$ for all $x \in \R^n$, $\supp f_i \subseteq B(0,R+1)$ and $f_i \rightarrow f$ in $L^\infty.$ Then the limit 
	\[
	L(f) = \lim_{k \rightarrow \infty} \int f \, d\mu_k
	\]
	exists and 
	\[
	L(f) = \lim_{i \rightarrow \infty} L(f_i). 
	\]
\end{lemma}

\begin{proof}
	Let $k \in \bN.$ For each $i,$ since $\mu_k$ is upper $C_0$-Ahlfors $d$-regular and $f_i \geq f,$ and for any $\gamma\in (0,1)$, we have 
	\begin{align*}
		\left| \int f \, d\mu_k - \int f_i \, d\mu_k \right| 
		&= \int f_i \, d\mu_k - \int f \, d\mu_k \\
		&=  \int (f_i -f +f ) \, d\mu_k - \int f \, d\mu_k \\
		& \stackrel{\eqref{e:subadd}}{\leq}
		\frac{1}{\gamma}\int (f_i -f  ) \, d\mu_k +\frac{\gamma}{1-\gamma} \int f \, d\mu_k \\
		&\leq \ps{\frac{1}{\gamma}|| f_i - f ||_\infty +\frac{\gamma}{1-\gamma} ||f||_{\infty}}\mu_{k}(B(0,R+1))\\
		&\leq C_0(R+1)^d\ps{\frac{1}{\gamma}|| f_i - f ||_\infty +\frac{\gamma}{1-\gamma} ||f||_{\infty}}
	\end{align*}
	Taking $k \rightarrow \infty,$ we find 
	\begin{multline*}
		L(f_i) - C_0(R+1)^d\ps{\frac{1}{\gamma}|| f_i - f ||_\infty +\frac{\gamma}{1-\gamma} ||f||_{\infty}}  \leq \liminf_{k \rightarrow \infty} \int f \, d\mu_k \\
		\leq \limsup_{k \rightarrow \infty} \int f \, d\mu_k 
		\leq \limsup_{k \rightarrow \infty} \int f_i \, d\mu_k 
		\leq L(f_i).
	\end{multline*}
	Since $L(f_i)$ is a monotone decreasing sequence of non-negative real numbers, $\lim_{i \rightarrow \infty} L(f_i)$ exists. Hence, taking $i \rightarrow \infty$, it follows that 
	\begin{multline*}
		\lim_{i\rightarrow \infty} L(f_i) - C_0(R+1)^d\frac{\gamma}{1-\gamma} ||f||_{\infty} \leq \liminf_{k \rightarrow \infty} \int f \, d\mu_k
		\leq \limsup_{k \rightarrow \infty} \int f \, d\mu_k \\
		\leq  \limsup_{k \rightarrow \infty} \int f_i \, d\mu_k 
		\leq \lim_{i\rightarrow \infty} L(f_i).
	\end{multline*}
	Since $\gamma\in (0,1)$ is arbitrary, this implies the limit $L(f)$ exists and equals the desired quantity. 
\end{proof}

\begin{lemma}\label{l:linear}
	The functional $L$ is linear on $C_0^+(\R^n).$ 
\end{lemma}

\begin{proof}
	Let $f,g \in C_0^+(\R^n).$ Since $D'$ is dense in $C_0^+(\R^n),$ we can find sequence of function $\{\tilde{f}_i\}$ and $\{\tilde{g}_i\}$ in $D'$ such that 
	\begin{align}\label{e:L-infty}
		|| f - \tilde{f}_i||_\infty \leq 3^{-i} \quad \text{and} \quad || g - \tilde{g}_i||_\infty \leq 3^{-i}.
	\end{align}
	Let $R_f$ and $R_g$ positive integers such that $\supp f \subseteq B(0,R_f)$ and $\supp g \subseteq B(0,R_g).$ Recall that functions in $D'$ are of the form $P\phi_j$ for some polynomial $P$ with rational coefficients and a bump function $\phi_j$ equal 1 on $B(0,j)$ with support in $B(0,j+1)$. Thus, without loss of generality we can assume $\tilde{f}_i = P_i^f \phi_{R_f}$ and $\tilde{g}_i = P_i^g \phi_{R_g}$ where $P_i^f,P_i^g$ are non-negative polynomials with rational coefficients such that
	\[||f  - P_i^f||_{L^\infty(B_{R_f})} \leq 3^{-i} \quad \text{and} \quad ||g - P_i^g||_{L^\infty(B_{R_g})} \leq 3^{-i}. \]
	
	We plan to modify the $\tilde{f}_i$ and $\tilde{g}_i$ so that they monotonically decrease to $f$ and $g$ respectively. For each $i,$ define 
	\[ f_i = \tilde{f}_i + 2 \cdot 3^{-i} \phi_{R_f} \quad \text{and} \quad g_i = \tilde{g}_i + 2 \cdot 3^{-i}\phi_{R_g}. \]
	We first consider the $f_i.$ We still have that $f_i \rightarrow f$ and $\supp f_i \subseteq B(0,R_f +1)$. Outside of $B(0,R_f)$, $f = 0 \leq f_i .$ For $x \in B(0,R_f)$ and $i \in \bN,$ 
	\[  f(x) + 3^{-i} \phi_{R_f}(x) = f(x) + 3^{-i} \stackrel{\eqref{e:L-infty}}{\leq}  \tilde{f}_i(x) +2\cdot 3^{-i}\phi_{R_f}(x)  = f_i(x)  \]
	and
	\begin{align*}
		f_i(x) &= \tilde{f}_i(x) + 2\cdot 3^{-i} \phi_{R_f}(x)\leq (P_i^f\phi_{R_f})(x)  + 2\cdot 3^{-i} \phi_{R_f}(x) \\
		&\leq (f(x) + 3^{-i})\phi_{R_f}(x) + 2\cdot 3^{-i}\phi_{R_f}(x) = f(x) + 3^{-i +1}\phi_{R_f}(x).
	\end{align*}
	It follows that $f \leq f_i$ for each $i$ and 
	\[ f_{i+1} \leq f + 3^{-i}\phi_{R_f} \leq f_i .\]
	In summary the sequence $\{f_i\}$ satisfies the hypothesis of Lemma \ref{l:continuity} for $f$. The same is true of the sequence $\{g_i \}$ for $g$. It is not difficult to show that 
	\[ h_i = f_i + g_i \]
	satisfies the conditions of Lemma \ref{l:continuity} for $f + g.$ Then, since $L$ is linear on $D$, we have 
	\[
	L(f + g) = \lim_{i \rightarrow \infty} L ( h_i) = \lim_{i \rightarrow \infty} L(f_i) + \lim_{i \rightarrow \infty} L(g_i) = L(f) + L(g)
	\]
	which completes the proof.
\end{proof}

\begin{proof}[Proof of Lemma \ref{l:subseq}] Let $\{\mu_{k_j}\}$ be the subsequence defining $L$. By Lemma \ref{l:linear}, $L$ defines a linear functional on $C_0^+(\R^n).$ We claim $L$ defines a linear functional on $C_0(\R^n).$ By definition, for any $f \in C_0(\R^n)$ and $k \in \bN,$ we have 
	\[ \int f \, d\mu_k = \int f^+ \, d\mu_k - \int f^{-} \, d\mu_k. \]
	Hence, the limit 
	\[ L(f) = \lim_{j \rightarrow \infty} \int f \, d\mu_{k_j} \] 
	exists and
	\[
	L(f) = L(f^+) - L(f^{-})
	\]
	Suppose $\phi,\vp \in C_0(\R^n)$. Then $\phi^+,\phi^{-},\vp^{+}$ and $\vp^-$ are in $C_0^+(\R^n).$  Observe that we can write
	\[
	(\phi+\vp)^+ - (\phi+\vp)^- = \phi + \vp = (\phi^+ + \vp^+) - (\phi^- + \vp^-),
	\]
	and after rearranging 
	\[ (\phi+\vp)^+ + (\phi^- + \vp^-) =  (\phi^+ + \vp^+) +   (\phi+\vp)^-. \]
	Taking $L$ on both sides and using linearity of $L$ on $C_0^+(\R^n)$, we have 
	\[ L((\phi+\vp)^+ ) + L( \phi^- + \vp^- ) = L(\phi^+ + \vp^+) + L((\phi+\vp)^-).\]
	Rearranging once more gives
	\begin{align}\label{e:1}
		L((\phi+\vp)^+ ) - L((\phi+\vp)^-) =  L(\phi^+ + \vp^+)  - L( \phi^- + \vp^- ).
	\end{align}
	Using \eqref{e:1}, linearity on $C_0(\R^n)$ follows since
	\begin{equation*}
		\begin{aligned}
			L(\phi+\vp) &= L((\phi+\vp)^+) - L((\phi+\vp)^-) \\
			&\stackrel{\eqref{e:1}}{=} L(\phi^+ + \vp^+) - L(\phi^- + \vp^-) \\
			&= L(\phi^+) - L(\phi^-) + L(\vp^+) - L(\vp^-) \\
			&=L(\phi) + L(\vp). 
		\end{aligned}
	\end{equation*}
	Now, since $L$ is linear on $C_0(\R^n),$ by the Riesz Representation Theorem we can find a Radon measure $\mu$ such that 
	\[
	\int \phi \, d\mu = \lim_{j \rightarrow \infty} \int \phi \, d\mu_{k_j}
	\]
	for all $\phi \in C_0(\R^n)$ as required. 
\end{proof}

\def\cprime{$'$}

\end{document}